\documentclass[12pt]{article}
\usepackage{epsfig,amsmath,amsthm,amssymb,graphics,amsfonts,psfrag}
\usepackage[letterpaper,margin=1in]{geometry}

\theoremstyle{plain}
\newtheorem{thm}{Theorem}[section]
\newtheorem{cor}[thm]{Corollary}
\newtheorem{prop}[thm]{Proposition}

\theoremstyle{definition}

\newtheorem{ex}[thm]{Example}

\def\R{\mathbb{R}}
\def\C{\mathbb{C}}
\def\N{\mathbb{N}}
\def\Z{\mathbb{Z}}
\def\E{\mathbb{E}}
\def\I{\infty}

\def\cF{\mathcal{F}}

\newcommand{\be}{\begin{equation}}
\newcommand{\ee}{\end{equation}}
\newcommand{\bea}{\begin{eqnarray}}
\newcommand{\eea}{\end{eqnarray}}
\newcommand{\beann}{\begin{eqnarray*}}
\newcommand{\eeann}{\end{eqnarray*}}
\newcommand{\benn}{\begin{equation*}}
\newcommand{\eenn}{\end{equation*}}

\def\ra{\rightarrow}
\def\I{\infty}

\begin{document}
 
\title{Nonlocal Generalized Models of Predator-Prey Systems}
\author{Christian Kuehn and Thilo Gross\thanks{Max Planck Institute for Physics of Complex Systems, 01187 Dresden, Germany}}

\maketitle

\begin{abstract}
The method of generalized modeling has been applied successfully in many different contexts, particularly in ecology and systems biology. It can be used to analyze the stability and bifurcations of steady-state solutions. Although many dynamical systems in mathematical biology exhibit steady-state behaviour one also wants to understand nonlocal dynamics beyond equilibrium points. In this paper we analyze predator-prey dynamical systems and extend the method of generalized models to periodic solutions. First, we adapt the equilibrium generalized modeling approach and compute the unique Floquet multiplier of the periodic solution which depends upon so-called generalized elasticity and scale functions. We prove that these functions also have to satisfy a flow on parameter (or moduli) space. Then we use Fourier analysis to provide computable conditions for stability and the moduli space flow. The final stability analysis reduces to two discrete convolutions which can be interpreted to understand when the predator-prey system is stable and what factors enhance or prohibit stable oscillatory behaviour. Finally, we provide a sampling algorithm for parameter space based on nonlinear optimization and the Fast Fourier Transform which enables us to gain a statistical understanding of the stability properties of periodic predator-prey dynamics. 
\end{abstract}


%%%%%%%%%%%%%%%%%%%%%%%%%%%%%%%%%%%%%%%%%%%%%%%%%%%%%%%%%%%%%%%%%%%%%%%%%%%%%%%%%%%%

\section{Introduction}
\label{sec:intro}

Predator-prey systems have been a cornerstone in mathematical biology for many decades  \cite{Berryman}. Standard textbooks on dynamical systems, differential equations and ecology provide a plethora of models that aim at capturing the interaction between a predator population $Y$ and a prey population $X$. Examples for modeling the situation by ordinary differential equations (ODEs) are \cite{Braun,BrauerCastillo-Chavez}
\benn
\text{(LV)}\left\{\begin{array}{lcl}
X'&=&p_1X-p_2XY, \\
Y'&=&p_3XY-p_4Y, \\
\end{array}\right. \qquad 
\text{(RM)}\left\{\begin{array}{lcl}
X'&=&k_1X-k_2X^2-k_3\frac{XY}{k_4+X}, \\
Y'&=&k_5\frac{XY}{k_4+X}-k_6Y, \\
\end{array}\right.
\eenn
where $p_i$, $k_i$ are parameters and (LV) are the Lotka-Volterra equations and (RM) is the Rosenzweig-MacArthur model \cite{Kot}. These two models are the most common examples of a large class of different models of the form
\be
\label{eq:gm_local_intro}
\begin{array}{lcl}
X'&=& S(X)-G(X,Y)\\
Y'&=& \alpha G(X,Y)-M(Y)\\
\end{array}
\ee 
where $\alpha>0$ is a parameter describing biomass conversion efficiency and the functions $S$, $G$ and $M$ represent prey growth, predation, and predator mortality, respectively. Because the parameter $\alpha$ can always be removed by scaling the variable $Y$ and re-labelling the functions we will always assume $\alpha=1$ from now on. 

Generalized models \cite{GrossFeudel,KuehnSiegmundGross} directly work with the formulation \eqref{eq:gm_local_intro} without specifying functional forms for $S$, $G$ and $M$. Previous works on generalized models \cite{Gross1,GrossBaurmannFeudelBlasius,GehrmannDrossel,ZumsandeStiefsSiegmundGross} focused on analyzing the dynamics close to stationary states. Beyond the structure of the equations \eqref{eq:gm_local_intro} this analysis requires only the assumption that steady states exist in the class of models under consideration. The central idea of generalized modeling is to parametrize all possible Jacobians that can be encountered in steady states in the class of systems under consideration. Using a specific renormalization procedure, one can define parameters that are easily interpretable (and often also directly measurable) in the context of an application. Applications of generalized models to ecology can be found in \cite{GrossFeudel2,GrossEbenhoehFeudel,GrossEbenhoehFeudel1,GrossFeudel,GrossFeudel1,BaurmannGrossFeudel,StiefsGrossSteuerFeudel,GrossRudolfLevinDieckmann,StiefsvanVoornKooiFeudelGross,YeakelStiefsNovakGross}. Let us emphasize that we do not claim that stability results from generalized models have not been observed before in some specific models; in fact, the literature on stability of planar-predator prey systems is very large. For instance, questions of local and global stability have been investigated in various predator-prey systems \cite{Goh,Hastings2,Wolkowicz,MoghadasAlexander,LiouCheng}.\\

In the present paper we go beyond the previous analysis and study nonstationary dynamics in the context of generalized modeling. We extend the theory of generalized models to arbitrary periodic solutions in the context of the predator-prey system \eqref{eq:gm_local_intro}. We show that this mathematical extension of generalized models yields several new phenomena in comparison to generalized models for steady states. For example, we can define time-periodic generalized parameters of the predator-prey model and we prove that these functions obey a system of ODEs (a flow on moduli space). Using Floquet theory \cite{Chicone} and Fourier analysis \cite{Katznelson} we derive analytical conditions for the solvability of the moduli space flow and obtain an analytical stability formula. In this context, a main result is that the stability formula for periodic solutions only depends on two constants that can be calculated via a discrete convolution. Using this formula we can identify parameters and conditions that enhance the stability of predator-prey cycles. Furthermore, we develop an algorithmic approach to sample the function space of parameters by solving an auxiliary optimization problem, which will be instrumental for future applications to larger systems.\\

The paper is structured as follows: In Section \ref{sec:background}, we recall the necessary tools from steady-state generalized models, Floquet theory, and Fourier analysis. In Section \ref{sec:gm_nonlocal}, we calculate the generalized vector field for non-equilibrium solutions. In Section \ref{sec:mod_flow}, we derive the flow on moduli space. In Section \ref{sec:specific}, we compute the generalized scale and elasticity functions for several specific functional forms to gain a better understanding how generalized and specific models link up. In Section \ref{sec:Fourier_decomp}, we use tools from Fourier analysis to derive algebraic conditions from the moduli space flow. In Section \ref{sec:stab1}, we provide an analytical stability analysis of periodic orbits in generalized predator-prey models. Using this result we can identify which situations increase or decrease stability and interpret the results in an ecological context. In Section \ref{sec:sampling}, we develop a sampling technique for generalized scale and elasticity functions that is based on solving an auxiliary optimization problem and the Fast Fourier Transform. We also use this sampling-based approach to improve our understanding of stabilizing and destabilizing factors of the predator-prey system. In Section \ref{sec:extensions}, we conclude with a brief summary and outline the large range of applications and theoretical challenges that can be found in non-equilibrium generalized models.

\section{Background}
\label{sec:background}

In this section we introduce essential tools and techniques that will be used throughout this work. Further, we use the opportunity to fix the notation. Below, we denote a general ordinary differential equation ODE by
\be
\label{eq:ODE}
\frac{dZ}{dt}=Z'=F(Z), \qquad \text{for $Z\in\R^N$}
\ee 
and assume always that $F$ is sufficiently smooth. In the following we are going to recall the necessary tools from steady state generalized models, Floquet theory and Fourier analysis.

\subsection{Generalized Models}
\label{sec:gm_local}

Let us start by reviewing generalized modeling \cite{GrossFeudel} for ODEs with equilibrium points. A detailed mathematical approach to generalized models can be found in \cite{KuehnSiegmundGross}. For the present discussion we restrict ourselves to review generalized models in the context of a planar predator-prey system \cite{GrossFeudel1}. Such systems describe the interaction of a population of prey X and a population of predators Y. The prey population grows at rate $S(X)$, predation occurs at rate $G(X,Y)$ and natural mortality of the predator at rate $M(Y)$. Denoting the prey density as X and predator density as Y we capture the dynamics by 
\be
\label{eq:gm_local}
\begin{array}{lcl}
X'&=& S(X)-G(X,Y),\\
Y'&=& G(X,Y)-M(Y),\\
\end{array}
\ee 
where $S$, $M\in C^r(\R^+,\R^+)$ and $G\in C^r(\R^+\times \R^+,\R^+)$ are sufficiently smooth functions. Generalized modeling assumes that \eqref{eq:gm_local} admits an equilibrium point $(X,Y)=(X^*,Y^*)\in\R^+\times \R^+$.\\

We normalize the equilibrium defining new coordinates
\be
\label{eq:gm_local_normalize}
x:=\frac{X}{X^*}\qquad \text{and}\qquad y:=\frac{Y}{Y^*}.
\ee
This transformation moves the equilibrium to $(x,y)=(1,1)$. The next step is to normalize the rate functions
\be
\label{eq:gm_local_rates}
s(x):=\frac{S(X^*x)}{S(X^*)},\qquad g(x,y):=\frac{G(X^*x,Y^*,y)}{G(X^*,Y^*)}, \qquad m(y):=\frac{M(Y^*y)}{M(Y^*)}. 
\ee
A direct substitution of \eqref{eq:gm_local_normalize}-\eqref{eq:gm_local_rates} into \eqref{eq:gm_local} gives 
\be
\label{eq:gm_local1}
\begin{array}{lcl}
x'&=& \frac{S(X^*)}{X^*}s(x)-\frac{G(X^*,Y^*)}{X^*}g(x,y),\\
y'&=& \frac{G(X^*,Y^*)}{Y^*}g(x,y)-\frac{M(Y^*)}{Y^*}m(y),\\
\end{array}
\ee 
where the prefactors of the form $S(X^*)/X^*$, $G(X^*,Y^*)/X^*$, etc. represent normalized fluxes in the steady state and are also called scale parameters
\be
\label{eq:gm_local_scale_ps}
\beta_s:=\frac{S(X^*)}{X^*},\quad \beta_1:=\frac{G(X^*,Y^*)}{X^*},\quad \beta_2:=\frac{G(X^*,Y^*)}{Y^*},\quad \beta_m:=\frac{M(Y^*)}{Y^*}.
\ee
Since $(x,y)=(1,1)$ is an equilibrium point we know that the following holds:
\be
\label{eq:gm_local_eq_cond}
\begin{array}{lclcl}
0&=&\frac{S(X^*)}{X^*}s(1)-\frac{G(X^*,Y^*)}{X^*}g(1,1)&=&\beta_s-\beta_1,\\
0&=&\frac{G(X^*,Y^*)}{Y^*}g(1,1)-\frac{M(Y^*)}{Y^*}m(1)&=&\beta_2-\beta_m.\\
\end{array}
\ee
Therefore \eqref{eq:gm_local1} can be re-written as
\be
\label{eq:gm_local2}
\begin{array}{lcl}
x'&=& \beta_1(s(x)-g(x,y)),\\
y'&=& \beta_2(g(x,y)-m(y)).\\
\end{array}
\ee 
The Jacobian at the equilibrium $(x,y)=(1,1)$ is then given by
\bea
\label{eq:local_Jac}
J(1,1)&=&\left(
\begin{array}{cc}
\beta_1~\partial_x [s(x)- g(x,y)]|_{(x,y)=(1,1)} & -\beta_1~\partial_y [g(x,y)]|_{(x,y)=(1,1)}\\ 
\beta_2~\partial_x [g(x,y)]|_{(x,y)=(1,1)} & \beta_2~\partial_y[g(x,y)- m(y)]|_{(x,y)={1,1}}\\ 
\end{array}
\right)\\
&=:&
\left(
\begin{array}{cc}
\beta_1[s_x-g_x] & -\beta_1g_y\\ 
\beta_2g_x & \beta_2[g_y-m_y]\\ 
\end{array}
\right)
\eea
where $\partial_x$, $\partial_y$ denote partial derivatives and we refer to the constants 
\be
\label{eq:gm_local_elasticities}
\begin{array}{lcl}
s_x=\partial_x(s(x))|_{x=1},&\quad &g_x=\partial_x(g(x,y))|_{(x,y)=(1,1)},\\
g_y=\partial_y(g(x,y))|_{(x,y)=(1,1)}, &\quad& m_y=\partial_y(m(y))|_{y=1},\\ 
\end{array}
\ee
as elasticities. The scale parameters and elasticities are also referred to as generalized parameters.\\

In the following we will use the insight that every power law function corresponds to an elasticity that is identical to the exponent of the power law. For example, if we assume that the mortality $M(Y)$ is a linear function $M(Y)=KY$ then we find 
\benn
m_y=\partial_y\left(\frac{M(Y^*y)}{M(Y^*)}\right)|_{y=1}=\partial_y\left(\frac{KY^*y}{KY^*}\right)|_{y=1}=1.
\eenn
Hence we can relate the growth properties of the unspecified functions forms to the elasticities.

The stability of the equilibrium $(x,y)=(1,1)$ can be inferred from the eigenvalues of $J(1,1)$ and hence only depends on the generalized parameters. This admits a bifurcation analysis of all steady state models of the form \eqref{eq:gm_local} in generalized parameter space. Despite the large class of models that one treats simultaneously it is often easy to interpret scale parameters and elasticities in applications \cite{GrossFeudel}. Thereby a generalized model enables us to draw conclusions about a whole class of differential equations, for further examples see \cite{GrossRudolfLevinDieckmann,StiefsvanVoornKooiFeudelGross,SteuerGrossSelbigBlasius,ReznikSegre,GehrmannDrossel}.

We note that generalized modeling can also be applied to equilibria for delay equations \cite{HoefenerSethiaGross,KuehnSiegmundGross}, spatially homogeneous states for partial differential equations \cite{BaurmannGrossFeudel} and to stochastic differential equations \cite{KuehnSiegmundGross}.

\subsection{Floquet Theory}
\label{sec:Floquet}

For analyzing the stability of periodic solutions in GM we resort to the framework offered by Floquet Theory. Suppose \eqref{eq:ODE} has a period orbit $\gamma(t)=\gamma(t+T)$ with minimal period $T$. Let $\Sigma$ denote a suitable $(N-1)$-dimensional transversal section to $\Gamma$ and consider the associated Poincar\'{e} map $P:\Sigma\ra \Sigma$. 
This map has a fixed point $X_\gamma\subset \Sigma$ associated to the periodic orbit $\gamma$ i.e. $P(X_\gamma)=X_\gamma$. Recall \cite{Chicone2,Kuznetsov} that the stability of $\gamma$ is determined by the $N-1$ eigenvalues (or characteristic/Floquet multipliers) $\lambda_1, \ldots, \lambda_{N-1}$ of the matrix $DP(X_\gamma)$. If $|\lambda_j|<1$ for all $j\in\{1,\ldots,N-1\}$ then the periodic orbit is stable, if there exists $\lambda_j$ such that $|\lambda_j|>1$ then the orbit is unstable and eigenvalues with $|\lambda_j|=1$ signal bifurcations under parameter variation. We can study the stability of $\gamma$ by considering the non-autonomous linear variational equation
\be
\label{eq:var}
v'=DF(\gamma(t))v=:A(t)v
\ee
where $A(t)$ is periodic. An $N\times N$ matrix $M(t)$ that satisfies
\be
\label{eq:var1}
M'=A(t)M\qquad \text{ with $M(0)=\text{Id}$} 
\ee
is called the fundamental matrix solution of $\eqref{eq:var}$. The constant matrix $M(T)$ is called the monodromy (or circuit) matrix. It has eigenvalues
\benn
1, \lambda_1, \lambda_2, \ldots, \lambda_{N-1}
\eenn 
where the trivial eigenvalue $1$ is associated to the direction tangent to the periodic orbit that links the variational equation to the Poincar\'{e} map $P$. Furthermore, the Liouville formula
\be
\label{eq:Liouville}
\lambda_1\lambda_2\cdots\lambda_{N-1}=\det M(T)=\exp\left(\int_0^T Tr(A(t))dt\right)
\ee 
holds. Floquet's theorem states that there exists a $T$-periodic coordinate change $C(t)$ and a constant matrix $R$ such that 
\benn
M(t)=C(t)e^{tR}.
\eenn
Since $M(0)=\text{Id}$ it follows that $C(0)=C(T)=\text{Id}$ and we find that the monodromy matrix can be expressed as 
\benn
M(T)=e^{RT}.
\eenn
An elegant explicit formula for the Floquet multiplier from \eqref{eq:Liouville} is only available for $N=2$. In general the computation of Floquet multiplier thus requires numerical approaches, which typically start with computing the periodic solution with a suitable boundary value method such as collocation or finite differences \cite{Kuznetsov,Doedel_AUTO2007}. The variational equation \eqref{eq:var1} is solved on suitable sub-intervals of the periodic orbit discretization as an initial value problem to obtain $M(T)$. The eigenvalues of $M(T)$ are then obtained yielding the Floquet multipliers. Although, in certain circumstances, such as large multipliers, the computation can be numerically problematic \cite{FairgrieveJepson,Lust}. 

Let us point out that Floquet theory has not been widely applied in the context of ecology \cite{Klausmeier} although it is a standard tool in the mathematical theory of dynamical systems \cite{Chicone2}. Klausmeier \cite{Klausmeier} suggests that ``Floquet theory [is] a useful tool for studying the effects of temporal variability on ecological system''. In the context of our approach, Floquet theory is not only a tool for a particular model but we will also show that it nicely extends to generalized models.

\subsection{Fourier Series}
\label{sec:Fourier}

Since we work with periodic solutions to ODEs and also other time-dependent periodic functions we briefly recall basic facts about Fourier series to fix normalization constants and notation. Assume that $f:\R\ra \R$ is $T$-periodic so that we can identify the domain of $f$ as the circle $\R/(T\Z)\cong S^1$. We can formally write the complex Fourier series $\cF[f]$ of $f$ as follows:
\be
\label{eq:cseries}
\cF[f](t)=\sum_{k=-\I}^\I \hat{f}(k) \exp\left(\frac{2\pi ikt}{T}\right)
\ee
where the Fourier coefficients $\hat{f}(k)$ are  
\benn
\hat{f}(k)=\frac1T \int_0^T f(s)\exp\left(-\frac{2\pi iks}{T}\right)ds.
\eenn
Observe that $\hat{\bar{f}}(k)=\overline{\hat{f}(-k)}$, where the overbar denotes complex conjugation. Further, $\hat{f}(0)=1/T\int_0^T f(t)dt$ is the time average of the periodic function. The convergence question $\cF[f](t)\ra f(t)$ is extremely intricate depending on the properties of $f$ \cite{Zygmund,Katznelson}. In the following, all functions we are going to approximate by Fourier series will be in $C^r(S^1,\R)$ for some sufficiently large $r$ or even $r=\I$. In this case, uniform convergence is immediate. A very important practical result in this context is to control the Fourier coefficients.

\begin{thm}[see \cite{Katznelson}]
\label{thm:Fc_decay}
If $f\in C^r(S^1,\R)$ then $|\hat{f}(k)|=\mathcal{O}(k^{-r})$ as $|k|\ra \I$. 
\end{thm}

Theorem \ref{thm:Fc_decay} is a version of the Riemann-Lebesgue Lemma for smooth functions and can provide an extremely rapid decay of the Fourier coefficients. This justifies (for the smooth case!) dropping higher-order terms $|k|>\kappa$ for some rather small suitable $\kappa\in \N$. The remaining sum is expected to be a good approximation to the original periodic function $f$. We write
\benn
\cF_\kappa[f](t):=\sum_{|k|\leq \kappa} \hat{f}(k)\exp\left(\frac{2\pi ikt}{T}\right) \approx f(t).
\eenn 
We remark that it can be convenient to re-write the complex Fourier series \eqref{eq:cseries} as a real Fourier series
\benn
\cF[f](t)=\frac{a_0}{2}+\sum_{k=1}^\I\left[a_k\cos\left(\frac{2\pi kt}{T}\right)+b_k\sin\left(\frac{2\pi kt}{T}\right)\right] 
\eenn
where the real Fourier coefficients relate to the complex ones by
\benn
\hat{f}(k)=\frac{1}{2}(a_k-ib_k) \qquad \text{and}\qquad \hat{f}(-k)=\frac{1}{2}(a_k+ib_k)
\eenn
for $k\in\N_0$. Another important tool in Fourier analysis we will need are convolutions. Recall that the discrete convolution of two periodic functions $f$ and $g$ is defined as 
\benn
(\hat{f}\ast\hat{g})(n)=\sum_{k=-\I}^\I \hat{f}(k)\hat{g}(n-k).
\eenn
Obviously the convolution operator `$\ast$' is associate, commutative and distributive. 

\section{Non-Equilibrium Planar Predator-Prey Systems}
\label{sec:gm_nonlocal}

We return to the planar predator-prey system \eqref{eq:gm_local} from Section \eqref{sec:gm_local} given by
\be
\label{eq:GF}
\begin{array}{lcl}
X'&=&S(X)-G(X,Y),\\
Y'&=&G(X,Y)-M(Y).\\
\end{array}
\ee
Denote the vector field of \eqref{eq:GF} by $F(X,Y)$. The vector field is only considered on the first (positive) quadrant $F:\R^+\times \R^+\ra \R^2$ as predator-prey densities are assumed to be non-negative. 

We want to analyze the class of vector fields \eqref{eq:GF} under the assumption that it admits a non-equilibrium orbit that is bounded as $|t|\ra \I$. From an ecological point of view the most interesting case are limit cycles, so-called predator-prey cycles. We assume that \eqref{eq:GF} has a periodic orbit $\gamma(t)=(\gamma_1(t),\gamma_2(t))$ with period $T$. The definition of the model implies that $\gamma_i>0$ for $i\in\{1,2\}$ and all $t$. 

In the following, we are going to slightly extend the notation employed already in Section \eqref{sec:gm_local} by re-using names for variables and generalized parameters. As in the case of equilibria one can consider a normalizing coordinate change 
\benn
x:=\frac{X}{\gamma_1} \qquad \text{and} \qquad y:=\frac{Y}{\gamma_2}
\eenn  
which maps the periodic orbit to the point $(x,y)=(1,1)=:1$. The ODEs \eqref{eq:GF} and the product rule imply
\beann
X'&=&x'\gamma_1+x \gamma_1'=x'\gamma_1+xF_1(\gamma)= S(x\gamma_1)-G(x\gamma_1,y\gamma_2),\\
Y'&=&y'\gamma_2+y \gamma_2'=y'\gamma_2+yF_2(\gamma)= G(x\gamma_1,y\gamma_2)-M(y\gamma_2).
\eeann
Therefore the new equations can be written as
\be
\label{eq:GF1}
\begin{array}{lcl}
x'&=&\frac{1}{\gamma_1}\left(S(x\gamma_1)-G(x\gamma_1,y\gamma_2)-xF_1(\gamma) \right)\\
&=&\frac{1}{\gamma_1}\left(S(x\gamma_1)-G(x\gamma_1,y\gamma_2)-x(S(\gamma_1)-G(\gamma)) \right),\\
y'&=&\frac{1}{\gamma_2}\left(G(x\gamma_1,y\gamma_2)-M(y\gamma_2)-yF_2(\gamma) \right)\\
&=&\frac{1}{\gamma_2}\left(G(x\gamma_1,y\gamma_2)-M(y\gamma_2)-y(G(\gamma)-M(\gamma_2)) \right).\\
\end{array}
\ee
In analogy to the equilibrium case we introduce normalized functions
\be
\label{eq:po_funcs}
s(x):=\frac{S(x\gamma_1)}{S(\gamma_1)}, \qquad g(x):=\frac{G(x\gamma_1,y\gamma_2)}{G(\gamma_1,\gamma_2)}, \qquad m(y):=\frac{M(y\gamma_2)}{M(\gamma_2)}. 
\ee
and define the scale parameters
\be
\label{eq:po_scales}
\begin{array}{lcl}
\beta_s(t):=\frac{S(\gamma_1(t))}{\gamma_1(t)},&\quad& \beta_1(t):=\frac{G(\gamma_1(t),\gamma_2(t))}{\gamma_1(t)},\\
\beta_2(t):=\frac{G(\gamma_1(t),\gamma_2(t))}{\gamma_2(t)},&\quad& \beta_m(t):=\frac{M(\gamma_2(t))}{\gamma_2(t)}.\\
\end{array} 
\ee
which are now time-dependent $T$-periodic scale functions. We will often suppress the time-dependence in the notation and just write, for instance, $\beta_s$ instead of $\beta_s(t)$. Using \eqref{eq:po_funcs}-\eqref{eq:po_scales} in \eqref{eq:GF1} we find
\be
\label{eq:GF2}
\begin{array}{lcl}
x'&=& \beta_s[s(x)-x] - \beta_1[g(x,y)-x],\\
y'&=& \beta_2[g(x,y)-y] - \beta_m[m(y)-y].\\
\end{array}
\ee
For applying Floquet theory we linearize \eqref{eq:GF2} around the limit cycle which yields the matrix
\benn
A(1;t)=\left(
\begin{array}{cc}
\beta_s[(\partial_xs)(1)-1] - \beta_1[(\partial_xg)(1)-1] & -(\partial_yg)(1) \\
(\partial_x g)(1) & \beta_2[(\partial_yg)(1)-1] - \beta_m[(\partial_ym)(1)-1]\\
\end{array}
 \right).
\eenn
We can re-write $A(1;t)$ in terms of the more familiar elasticities, leading to
\benn
A(1;t)=
\left(
\begin{array}{cc}
\beta_s(t)[s_x(t)-1] - \beta_1(t)[g_x(t)-1]) & -g_y(t) \\
g_x(t) & \beta_2(t)[g_y(t)-1] - \beta_m(t)[m_y(t)-1])\\
\end{array}
 \right)
\eenn
where the four time-dependent $T$-periodic elasticity functions are
\benn
s_x(t):=(\partial_xs)(1), \quad g_x(t):=(\partial_xg)(1), \quad g_y(t):=(\partial_yg)(1), \quad m_y(t):=(\partial_ym)(1).
\eenn
The periodicity and time-dependence becomes more apparent once we write out the detailed definitions, for example
\benn
s_x(t)=(\partial_xs)(1)=\partial_x\left(\frac{S(x\gamma_1)}{S(\gamma_1)}\right)|_{x=1}=\frac{\gamma_1S'(\gamma_1)}{S(\gamma_1)}.
\eenn
The previous calculations show that we can introduce replacements for the generalized parameters for equilibrium points in the context of periodic orbits. In particular, the scale parameters and elasticities become time-dependent and periodic. The term ``generalized functions'' is already used in a different context \cite{Zemanian}. Therefore, we refer to elasticity functions and scale functions directly. To analyze the stability of the periodic solution we use Floquet theory (see Section \ref{sec:Floquet}). For planar systems the stability of the periodic orbit is determined by computing the only non-trivial Floquet multiplier $\lambda$. Liouville's formula implies that
\bea
\label{eq:FM}
\lambda&=&\exp\left(\int_0^T \text{Tr}(A(1;t))dt\right)\nonumber\\
&=& \exp\left(\int_0^T \beta_s (s_x-1)-\beta_1 (g_x-1) +\beta_2(g_y-1)-\beta_m(m_y-1)dt\right).
\eea
We can thus express the Floquet multiplier as a function depending on elasticity and scale functions. This is analogous to writing the eigenvalues of the Jacobian as functions of the generalized parameters in the equilibrium case.

\section{The Moduli Space Flow}
\label{sec:mod_flow}

In analogy to the generalized exploration of local dynamics, the stability of the limit cycle can be studied by assuming plausible values for the generalized parameters (here, scale and elasticity functions). The value of generalized models lies in their ability to cover the whole range of possibilities that are plausible in the system. For an unbiased analysis it is essential that we consider only those values of parameters that are consistent with the set up of the system. For instance, in case of equilibrium generalized models we must demand that the parameter values which we assume do not preclude the existence of an equilibrium solution in the class of systems. Likewise, only those scale and elasticity functions should be considered which are mutually consistent and thus could arise in at least one example system in the class of models under consideration. To understand this problem we briefly go back to the equilibrium scenario (see Section \ref{sec:gm_local}). Suppose we just choose a set of generalized parameters
\be
\label{eq:gm_local_choice}
\beta_1=\beta_1^*, \quad \beta_2=\beta_2^*,\quad s_x=s_x^*, \quad g_x=g_x^*, \quad g_y=g_y^*,\quad m_y=m_y^*,
\ee
where we assume that all parameters are positive. One natural question is if there exist specific functions $S$, $G$ and $M$ that lead to the generalized parameters \eqref{eq:gm_local_choice}. 

\begin{prop}
\label{prop:local_specific}
Suppose \eqref{eq:gm_local_choice} are given positive generalized parameters. Then there exist functions $S$, $G$, $M$ and an equilibrium $(X,Y)=(X^*,Y^*)$ for \eqref{eq:gm_local} so that \eqref{eq:gm_local_scale_ps} and \eqref{eq:gm_local_elasticities} hold i.e. there exists a differential equation of the form \eqref{eq:gm_local} that has the given set of generalized parameters.
\end{prop}

\begin{proof}
Pick $M(Y)=p_1Y^{m_y^*}$ for some $p_1\in\R^+$ then $m_y=\partial_y(M(Y^*y)/M(Y^*))|_{y=1}=m_y^*$. Similarly we pick $S(X)=p_2X^{s_x^*}$ and obtain $s_x=s_x^*$. Using a slight modification of this approach we define $G(X,Y)=X^{g_x^*}Y^{g_y^*}$ and get $g_y=g_y^*$ as well as $g_y=g_y^*$. We also must have $\beta_s=\beta_1=\beta_1^*$ and $\beta_m=\beta_2=\beta_2^*$ which translates into the conditions
\beann
\beta_1\stackrel{(C1)}{=}&p_1(X^*)^{s_x-1}\stackrel{(C2)}{=}& (X^*)^{g_x-1}(Y^*)^{g_y},\\
\beta_2\stackrel{(C3)}{=}&p_2(Y^*)^{m_y-1}\stackrel{(C4)}{=}& (X^*)^{g_x}(Y^*)^{g_y-1}.\\
\eeann
We can always choose $p_1$ and $p_2$ to satisfy $(C2)$ and $(C4)$. Then we can use $X^*$ and $Y^*$ to satisfy $(C1)$ and $(C3)$. The result follows.
\end{proof}

Although there are certainly many other ways for constructing functions that are consistent with a given set of scale and elasticity parameters, already the existence of one such set of functions proves that the assumed parameter values could be encountered in the class of models under consideration. This observation is of central importance for sampling procedures, by which high dimensional generalized models are typically analyzed \cite{SteuerGrossSelbigBlasius,GrossRudolfLevinDieckmann}. 

For non-equilibrium systems the situation is different since one has to ask whether a whole set of given functions 
\benn
\beta_s(t), \quad \beta_m(t), \quad \beta_1(t), \quad \beta_2(t), \quad s_x(t), \quad g_x(t), \quad g_y(t),\quad m_y(t),
\eenn
can potentially arise from a system of the form \eqref{eq:GF}.

\begin{thm}
\label{thm:mod_flow}
Suppose we are given elasticity functions $s_x$, $m_y$, $g_x$ and $g_y$ then the scale functions have to satisfy the following set of ODEs
\be
\label{eq:mod_flow}
\begin{array}{lcl}
\beta_s'&=& \beta_s(\beta_s-\beta_1)(s_x-1),\\
\beta_m'&=& \beta_m(\beta_2-\beta_m)(m_y-1),\\
\beta_1'&=&\beta_1((\beta_s-\beta_1)g_x-(\beta_s-\beta_1)+(\beta_2-\beta_m)g_y),\\
\beta_2'&=&\beta_2((\beta_2-\beta_m)g_y-(\beta_2-\beta_m)+(\beta_s-\beta_1)g_x).\\
\end{array}
\ee
\end{thm}

\begin{proof}
We start by deriving the equation for $\beta_s$. We know that $\beta_s=S(\gamma_1)/\gamma_1$ and direct differentiation with respect to time via the quotient and chain rules gives
\beann
\beta_s'&=&\frac{\gamma_1\gamma_1'S'(\gamma_1)-S(\gamma_1)\gamma_1'}{(\gamma_1)^2}\\
&=&\frac{\gamma_1\gamma_1'S'(\gamma_1)S(\gamma_1)}{(\gamma_1)^2S(\gamma_1)}-\frac{S(\gamma_1)\gamma_1'}{(\gamma_1)^2}
\eeann
Noting that $s_x=\gamma_1S'(\gamma_1)/S(\gamma_1)$ and using the definition of $\beta_s$ the equation transforms to 
\bea
\beta_s'&=&\frac{s_x\gamma_1'S(\gamma_1)}{(\gamma_1)^2}-\frac{S(\gamma_1)\gamma_1'}{(\gamma_1)^2} \nonumber\\
&=&\frac{s_x\gamma_1'\beta_s}{\gamma_1}-\frac{\beta_s\gamma_1'}{\gamma_1}=\beta_s(s_x-1)\frac{\gamma_1'}{\gamma_1}. \label{eq:thm1_proof}
\eea
Since $(\gamma_1,\gamma_2)$ is a trajectory of \eqref{eq:GF} we must have $\gamma_1'=S(\gamma_1)-G(\gamma_1,\gamma_2)$. This implies upon substitution into \eqref{eq:thm1_proof} that
\beann
\beta_s'&=& \beta_s(s_x-1)\frac{S(\gamma_1)-G(\gamma_1,\gamma_2)}{\gamma_1}\\
&=& \beta_s(s_x-1)(\beta_s-\beta_1)
\eeann
which is the first equation in \eqref{eq:mod_flow}. The calculation for $\beta_m'$ is similar. For $\beta_1'$ we find
\beann
\beta_1'&=& \frac{\gamma_1[G_x(\gamma_1,\gamma_2)\gamma_1'+G_y(\gamma_1,\gamma_2)\gamma_2']-\gamma_1'G(\gamma_1,\gamma_2)}{(\gamma_1)^2}\nonumber\\
&=& \frac{\gamma_1G_x(\gamma_1,\gamma_2)\gamma_1'G(\gamma_1,\gamma_2)}{G(\gamma_1,\gamma_2)(\gamma_1)^2}+\frac{\gamma_1G_y(\gamma_1,\gamma_2)\gamma_2'}{(\gamma_1)^2}\frac{G(\gamma_1,\gamma_2)\gamma_2}{G(\gamma_1,\gamma_2)\gamma_2}-\frac{\gamma_1'G(\gamma_1,\gamma_2)}{(\gamma_1)^2}\nonumber\\
&=& g_x\beta_1\frac{\gamma_1'}{\gamma_1}+g_y\beta_1\frac{\gamma_2'}{\gamma_2}-\beta_1\frac{\gamma_1'}{\gamma_1}\nonumber\\
&=& \beta_1(g_x(\beta_s-\beta_1)+g_y(\beta_2-\beta_m)-(\beta_s-\beta_1))
\eeann
The calculation for $\beta_2'$ is similar to the one for $\beta_1'$.
\end{proof}

The main conclusion is that the elasticities and scale functions which parametrize the ODE \eqref{eq:GF2} satisfy an ODE themselves. Because one often uses the terms ``parameters'' and ``moduli'' interchangeably, Theorem \ref{thm:mod_flow} implies that the time-dependent parameters of generalized models generate a flow on moduli space. The following remark describes the relevance of this viewpoint in some other research areas.\\

\textit{Remark:} The term "moduli space" is perhaps most commonly used in algebraic geometry which, broadly speaking, is the study of solutions of algebraic equations \cite{Harris,Hartshorne}. The solutions form algebraic varieties (e.g.{ }curves). Often suitable parametrized families of algebraic varieties again have the structure of an algebraic variety, where the latter object is the moduli space of parametrized families. The study of the geometry of moduli spaces has also been transported into different branches of physics such as quantum field theory \cite{Andrianopolietal}. In dynamical systems theory, a classical moduli space argument is made in the renormalization analysis of parametrized families one-dimensional maps \cite{GH}, where the renormalization transformation can be viewed as a map generating a dynamical system on moduli space. A very similar situation occurs for billiard dynamics where the so-called Teichm\"{u}ller flow on the space of lattices appears \cite{SmillieMSRI,MasurTabachnikov}.\\

We note that the positive quadrant is an invariant set for \eqref{eq:mod_flow} which means that this property lifts from the predator-prey family of vector fields to the moduli space. From Theorem \ref{thm:mod_flow} we can immediately infer a condition for the existence of a generalized model with given elasticities.

\begin{cor}
\label{cor:thm1}
Suppose $s_x$, $m_y$, $g_x$ and $g_y$ are given $T$-periodic elasticity functions with minimal period $T$. If \eqref{eq:mod_flow} has no $T$-periodic solutions then there exists no generalized model of the form \eqref{eq:GF} for the given elasticities. 
\end{cor}

Note that the existence of periodic solutions in Corollary \ref{cor:thm1} is only a necessary condition for the existence of a generalized model. We also observe that for the equilibrium case the conditions $\beta_1=\beta_s$ and $\beta_2=\beta_m$ give $\beta_s'=\beta_m'=\beta_1'=\beta_2'=0$ consistent with steady state generalized modeling. It is also interesting to ask what happens if we do not specify the elasticities. 

Taking the idea of deriving a differential equation one step further we consider 
\benn
s_x=\frac{\gamma_1S'(\gamma_1)}{S(\gamma_1)}\quad \Rightarrow s_x'=\frac{S(\gamma_1)[\gamma_1'S'(\gamma_1)+\gamma_1\gamma_1'S''(\gamma_1)]-\gamma_1S'(\gamma_1)\gamma_1'S'(\gamma_1)}{S(\gamma_1)^2}
\eenn
Applying similar substitutions as in the proof of Theorem \ref{thm:mod_flow} we obtain
\beann
s_x'&=&\frac{S(\gamma_1)[\gamma_1'S'(\gamma_1)+\gamma_1\gamma_1'S''(\gamma_1)]-\gamma_1S'(\gamma_1)\gamma_1'S'(\gamma_1)}{S(\gamma_1)^2}\\
&=&\frac{\gamma_1'}{\gamma_1}s_x+\frac{\gamma_1'S''(\gamma_1)}{\beta_s}-\frac{\gamma_1'}{\gamma_1}s_x^2\\
&=&s_x(s_x-1)(\beta_s-\beta_1)+\frac{1}{\beta_s}\gamma_1'S''(\gamma_1).
\eeann
Similar calculations can be carried out for $g_x'$, $g_y'$ and $m_y'$. These suggest that specifying a suitable scaled version of second partial derivatives of $S$, $G$ and $M$ will provide a system of eight ODEs. This procedure could be continued iteratively. It is interesting to note that closing a system of ODEs at a given order is a problem that also occurs in the context of moment closure for networks \cite{KeelingRandMorris,GrossDLimaBlasius} and for moment equations of stochastic differential equations \cite{Socha,Gardiner}. To get a better understanding of the stability of non-equilibrium generalized models and the flow on moduli space we proceed to consider a few typical specific functions $S$, $G$ and $M$ that appear in predator-prey models.

\section{Specific Functions}
\label{sec:specific}

In this section we calculate the generalized elasticity and scale functions for several well-known predator-prey models. All the model parameters $k_l$ (for $l\in \N$) we are going to use below are positive due to modeling considerations. We start with the growth of the prey $S(X)$. Typical choices are
\benn
\begin{array}{lclcl}
S(X)&=&k_1X&\quad& \text{(linear growth)},\\
S(X)&=&k_1X^p&\quad& \text{(power growth)},\\
S(X)&=&k_1X-k_2X^2&\quad& \text{(logistic growth)},\\
S(X)&=&k_1X(k_2-X)(X-k_3)&\quad& \text{(growth with strong Allee effect), $0<k_2<k_3$}.\\
\end{array} 
\eenn 
We start by looking at linear growth. We find
\benn
\beta_s=\frac{S(\gamma_1)}{\gamma_1}=k_1,\qquad s_x=\partial_x\left.\left(\frac{S(x\gamma_1)}{S(\gamma_1)}\right)\right|_{x=1}=1.
\eenn
For estimating the impact of linear growth on stability we consider the formula \eqref{eq:FM} and view it as a product of exponentials. The term involving $\beta_s$ and $s_x$ is
\be
\label{eq:stab_form_partS}
\exp\left(\int_0^T \beta_s [s_x-1] dt\right).
\ee
Therefore, a linear prey growth does not contribute to the non-trivial Floquet multiplier because $s_x=1$ and $\exp(\int_0^T 0dt)=1$. Different types of polynomial growth with a single term can be treated analogously since for $S(X)=k_1X^p$ we find
\benn
\beta_s=\frac{S(\gamma_1)}{\gamma_1}=k_1\gamma_1^{p-1},\qquad s_x=\partial_x\left.\left(\frac{S(x\gamma_1)}{S(\gamma_1)}\right)\right|_{x=1}=p.
\eenn
where the elasticity function coincides with the result for equilibrium generalized models. This allows us to write 
\benn
\beta_s[s_x-1]=k_1\gamma_1^{p-1}[p-1].
\eenn
Considering \eqref{eq:stab_form_partS} we find that increasing $p$ increases the Floquet multiplier and therefore has always a destabilizing effect, whereas decreasing $p$ has a stabilizing effect. For logistic growth we obtain 
\benn
\beta_s=\frac{S(\gamma_1)}{\gamma_1}=k_1-k_2\gamma_1,\qquad s_x=\partial_x\left.\left( \frac{S(x\gamma_1)}{S(\gamma_1)}\right)\right|_{x=1}=2 +\frac{k_1}{-k_1 + k_2 \gamma_1}.
\eenn
This implies
\be
\label{eq:S_logistic}
\beta_s[s_x-1]=(k_1-k_2\gamma_1)\left[2 +\frac{k_1}{-k_1 + k_2 \gamma_1}-1\right]=-k_2\gamma_1.
\ee
Considering \eqref{eq:stab_form_partS} again we find
\benn
\exp\left(\int_0^T \beta_s [s_x-1] dt\right)=\exp\left(\int_0^T -k_2\gamma_1 dt\right)=\exp\left(-k_2 \int_0^T \gamma_1 dt\right)
\eenn 
where the integral is positive because $k_2>0$. This means that increasing $k_2$ or increasing $\int_0^T \gamma_1 dt$ will promote stability as the Floquet multiplier will move closer to $0$. For logistic growth increasing $k_2$ corresponds to decreasing the carrying capacity $k_1/k_2$ of the population. This can be interpreted as a manifestation of the paradox of enrichment \cite{Rosenzweig,GrossEbenhoehFeudel1} which captures the observation that increasing the carrying capacity generally has a destabilizing effect on attractors observed in ecological models. Furthermore, the expression obtained for logistic grows permits discussion of the contribution of the shape of the limit cycle to stability. For $t\in[\delta_1,T-\delta_2]$, where $\delta_{1,2}>0$ are small, we find that
\be
\label{eq:small_prey}
0<\gamma_1(t)\ll1
\ee
which implies that the integral $\int_0^T\gamma_1(t)dt$ is small as well. Therefore limit cycles where the number of prey is extremely small for long times are not expected to be a basis for a stable ecosystem. For the Allee effect we find
\bea
\label{eq:S_Allee}
\beta_s[s_x-1]&=&k_1(k_2-\gamma_1)(\gamma_1-k_3)\left[1+\gamma_1\left(\frac{1}{\gamma_1-k_2}+\frac{1}{\gamma_1+k_3}\right)-1\right]\nonumber \\
&=& k_1\gamma_1(k_2+k_3-2\gamma_1).
\eea
Considering the contribution of this term to the Floquet multiplier yields
\beann
\exp\left(\int_0^T \beta_s [s_x-1] dt\right)&=&\exp\left(\int_0^T k_1\gamma_1(k_2+k_3-2\gamma_1) dt\right)\\
&=&\exp\left(\int_0^T k_1\gamma_1(k_2+k_3)dt-\int_0^T k_1 2\gamma_1^2 dt\right).
\eeann 
Increasing $k_2$ and/or $k_3$ will decrease stability. This is natural as these parameters represent the threshold to growth and the carrying capacity, providing another example for the paradox of enrichment. 

Note that the shape of the limit cycle can influence stability. In particular, the same conclusion to assumption \eqref{eq:small_prey} holds. In the case of the Allee effect the de-stabilization effect for long periods of low prey density even enters quadratically in the term $\int_0^T k_1 2\gamma_1^2 dt$. This confirms the intuitive conclusion that imposing a threshold to growth is a de-stabilizing factor for non-equilibrium systems when the prey density is small. 

We proceed to consider the mortality of the predator. A very common functional form used in a large number of models is so-called density independent (linear) mortality
\benn
M(Y)=k_1Y\qquad\Rightarrow \quad\beta_m[m_y-1]=k_1[m_y-1]=k_1[1-1]=0.
\eenn
Therefore, linear predator mortality has no effect on the stability of the periodic solution.\\

\begin{table}
\centering
\begin{tabular}{|lclcl|}
\hline
\multicolumn{5}{|c|}{Common Functional Forms} \\
\hline
$G(X,Y)$ & $=$ & $k_1 XY$ & $\quad$ & $\text{(Holling type I)}$,\\
$G(X,Y)$ & $=$ & $\frac{k_1 XY}{k_2+X}$ & $\quad$ & $\text{(Holling type II)}$,\\
$G(X,Y)$ & $=$ & $\frac{k_1 X^2Y}{k_2+X^2}$ &$\quad$& $\text{(Holling type III)}$,\\
$G(X,Y)$ &$=$& $\frac{k_1 X^2Y}{k_2+X+k_3X^2}$ & $\quad$& $\text{(Holling type IV)}$.\\
\hline
\multicolumn{5}{|c|}{$\quad$}\\
\multicolumn{5}{|c|}{Terms occurring in the Floquet multiplier $\ldots$ }\\ 
\hline
$\text{(Holling type I)} $   & $\quad$ & $\beta_1[g_x-1]$ & $=$ & $k_1\gamma_2[1-1]=0$,\\
                           &$\quad$& $\beta_2[g_y-1]$&$=$&$k_1\gamma_1[1-1]=0$,\\
$\text{(Holling type II)} $  &$\quad$& $\beta_1[g_x-1]$&$=$&$\frac{k_1\gamma_2}{k_2+\gamma_1}[\frac{k_2}{k_2+\gamma_1}-1]=-\frac{k_1\gamma_1\gamma_2}{(k_2+\gamma_1)^2}$,\\
                           &$\quad$& $\beta_2[g_y-1]$&$=$&$\frac{k_1\gamma_1}{k_2+\gamma_1}[1-1]=0$,\\
$\text{(Holling type III)}$  &$\quad$& $\beta_1[g_x-1]$&$=$&$\frac{k_1\gamma_1\gamma_2}{k_2+\gamma_1^2}[\frac{2k_2}{k_2+\gamma_1^2}-1]=\frac{k_1\gamma_1\gamma_2(k_2-\gamma_1^2)}{(k_2+\gamma_1^2)^2}$,\\
                           &$\quad$& $\beta_2[g_y-1]$&$=$&$\frac{k_1\gamma_1^2}{k_2+\gamma_1^2}[1-1]=0$,\\
$\text{(Holling type IV)}$   &$\quad$& $\beta_1[g_x-1]$&$=$&$\frac{k_1 \gamma_1 \gamma_2}{k_2+\gamma_1+k_3 \gamma_1^2}[\frac{2k_2+\gamma_1}{k_2+\gamma_1+k_3\gamma_1^2}-1]=\frac{k_1\gamma_1\gamma_2(k_2-k_3\gamma_1^2)}{(k_2+\gamma_1+k_3\gamma_1^2)^2}$,\\
                           &$\quad$& $\beta_2[g_y-1]$&$=$&$\frac{k_1\gamma_1^2}{k_2+\gamma_1+k_3\gamma_1^2}[1-1]=0$.\\
\hline
\end{tabular} 
\caption{\label{tab:tab2}Calculation of the scale and elasticity functions for the predation term $G(X,Y)$. The top panel lists four typical functional forms. In the bottom panel we calculate the terms occurring in the Floquet multiplier \eqref{eq:FM}.}
\end{table} 

The interaction term between prey and predator is usually the most complicated and debated choice for the model. Some common choices are considered in Table \ref{tab:tab2}. The observation that $\beta_2[g_y-1]$ vanishes for all functions considered in Table \ref{tab:tab2}, can be directly linked to the ecological assumption that predators hunt independently of each other. The functions that are therefore used in practice are generally linear in the density of predators and the impact of predator dependence on stability vanishes. The same assumption cannot generally be made for prey dependence of predation, leading to more complex expressions for the impact on stability. 

Therefore, we are going to make the assumptions
\be
\label{eq:gy1my1}
g_y=1\qquad \text{and}\qquad m_y=1
\ee
from now. Regarding the Floquet multiplier formula \eqref{eq:FM} the assumptions \eqref{eq:gy1my1} simplify the situation to investigating
\be
\label{eq:FM1}
\lambda=\exp\left(\int_0^T \beta_s (s_x-1)-\beta_1 (g_x-1)dt\right).
\ee
The influence of $g_x$ and $\beta_1$ on stability is not obvious since there is a non-trivial interaction with the shape of the limit cycle. The flow on moduli space given by \eqref{eq:mod_flow} simplifies to
\be
\label{eq:mod_flow1}
\begin{array}{lcl}
\beta_s'&=& \beta_s(\beta_s-\beta_1)(s_x-1),\\
\beta_1'&=&\beta_1((\beta_s-\beta_1)g_x-(\beta_s-\beta_1)+(\beta_2-\beta_m)),\\
\beta_2'&=&\beta_2(\beta_s-\beta_1)g_x.\\
\beta_m'&=& 0,\\
\end{array}
\ee
where we can view $\beta_m$ as a parameter and simply drop the last equation.\\ 

\begin{figure}[htbp]
\psfrag{a}{(a)}
\psfrag{b}{(b)}
\psfrag{c}{(c)}
\psfrag{d}{(d)}
\psfrag{e}{(e)}
\psfrag{f}{(f)}
\psfrag{t}{$t$}
\psfrag{el}{$el$}
\psfrag{beta}{$\beta$}
\psfrag{bs}{$\beta_s$}
\psfrag{b1}{$\beta_1$}
\psfrag{b2}{$\beta_2$}
\psfrag{X}{$X$}
\psfrag{Y}{$Y$}
\psfrag{ga}{$\gamma_i$}
	\centering
		\includegraphics[width=1\textwidth]{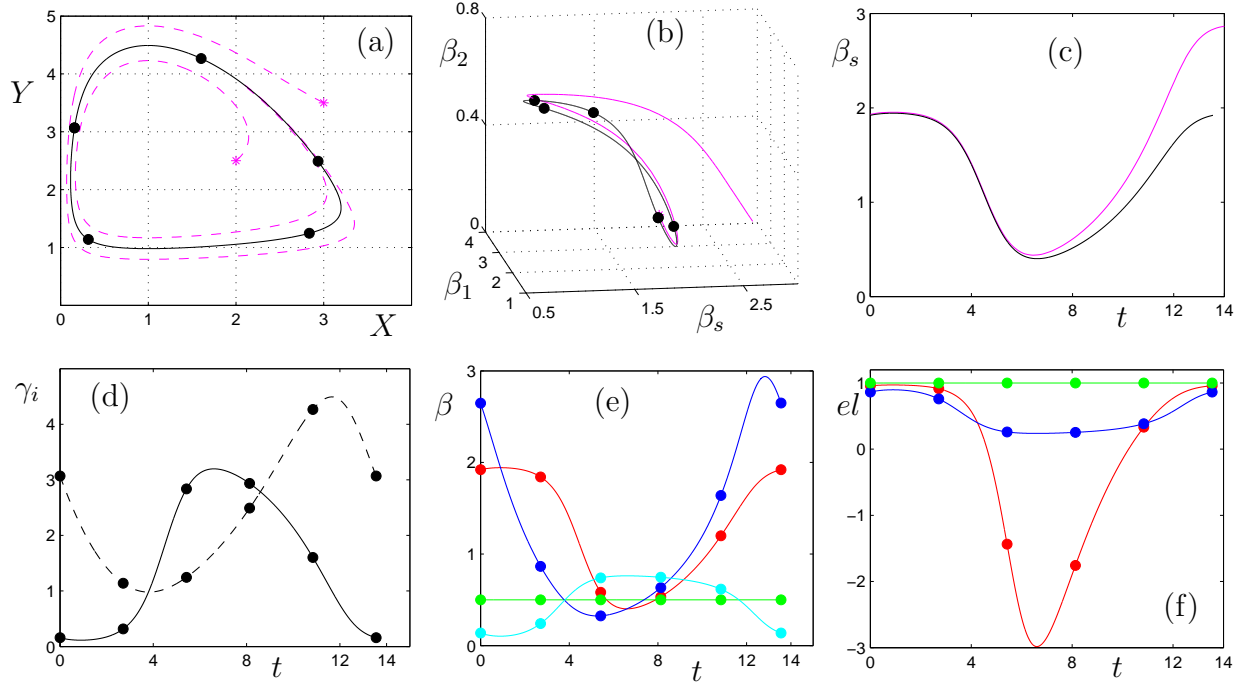}
	\caption{\label{fig:fig1}Dynamics in a specific example. (a) Stable periodic orbit $\gamma(t)$ of \eqref{eq:RM} (solid black) and two other trajectories (dashed magenta) with initial conditions marked by stars; the parameters are given in \eqref{eq:RM_para}. Five points (black dots) are shown on the limit cycle for orientation purposes which are equally space over one period. (b) Scale functions in moduli space (black) for $\gamma$ solving \eqref{eq:mod_flow1}; a trajectory (solid magenta) with slightly perturbed initial conditions is also shown where the same elasticities as for the periodic orbit were used for numerical integration. (c) Time series of $\beta_s(t)$ for part (b).(d) Time series $\gamma_1$ (solid black) and $\gamma_2$ (dashed black). (e) Scale functions associated to $\gamma$: $\beta_s(t)$ (red), $\beta_m(t)$ (green), $\beta_1(t)$ (blue) and $\beta_2(t)$ (cyan). (f) Elasticity function associated to $\gamma$: $s_x(t)$ (red), $g_x(t)$ (blue) and $g_y(t)$ (green); note that $m_y=1=g_y$. }
\end{figure}

\begin{ex}
\label{ex:ex1}
For gaining an intuitive understanding one can consider the flow on the moduli space in a specific example. The combination of logistic prey growth, Holling-type-II interaction and linear predator mortality gives us the Rosenzweig-MacArthur predator-prey model that can produce periodic solutions
\be
\label{eq:RM}
\begin{array}{lcl}
X'&=& k_1X-k_2X^2-k_3\frac{XY}{X+k_4},\\
Y'&=& k_3\frac{XY}{X+k_4}-k_5Y,\\
\end{array}
\ee
where we use the parameters 
\be
\label{eq:RM_para}
k_1=2,\qquad k_2=0.5,\qquad k_3=1,\qquad k_4=1, \qquad k_5=0.5.
\ee
Figure \ref{fig:fig1} shows that integrating a slightly perturbed initial condition trajectory does seem to diverge from the exact periodic solution in moduli space. Furthermore, even for a classical planar predator-prey system, the scale and elasticity functions are quite complicated for non-equilibrium solutions. In fact, prescribing the elasticities is much more difficult than just picking a set of fixed parameters for equilibrium generalized models.\\ 
\end{ex}

To verify the necessary condition from Corollary \ref{cor:thm1} for periodic solutions we must ask for solvability of the boundary value problem (BVP)
\be
\label{eq:mod_BVP}
\left\{\begin{array}{lcl}
\beta_s'&=& \beta_s(\beta_s-\beta_1)(s_x-1),\\
\beta_1'&=&\beta_1((\beta_s-\beta_1)g_x-(\beta_s-\beta_1)+(\beta_2-k_5)),\\
\beta_2'&=&\beta_2(\beta_s-\beta_1)g_x,\\
\beta(0)&=&\beta(T)~\text{ for $T>0$},\\ 
\end{array}\right.
\ee
where $\beta(t):=(\beta_s(t),\beta_1(t),\beta_2(t))$. It is well-known that BVPs can have one, many or no solutions \cite{AscherMattheijRussell}. Furthermore determining solvability conditions is usually not easy and even using numerical methods may be dangerous; for example, if a numerical algorithm fails to provide a solution to \eqref{eq:mod_BVP} this may just be due to the numerical problems that can arise when solving BVPs \cite{AscherMattheijRussell}.

\section{Fourier Decomposition}
\label{sec:Fourier_decomp}

The previous discussion of specific functions and the Rosenzweig-MacArthur model motivate the need for a more concrete version of the moduli space conditions \eqref{eq:mod_BVP} and of the Floquet multiplier \eqref{eq:FM1}. The natural step is to use a decomposition of the periodic functions into Fourier series; see Section \ref{sec:Fourier}. Using discrete convolution we can easily re-write the problem \eqref{eq:mod_BVP} on moduli space.

\begin{prop}
\label{prop:FC_sat}
Suppose we are given $T$-periodic elasticity functions $s_x$, $m_y$, $g_x$ and $g_y$. Then the Fourier coefficients of periodic scale functions have to satisfy the following set of algebraic equations 
\be
\label{eq:mod_flow_Fourier}
\begin{array}{lcl}
\frac{2\pi i k}{T}\hat{\beta}_s(k)&=& [\hat{\beta}_s\ast(\hat{\beta}_s-\hat{\beta}_1)\ast(\hat{s}_x-\hat{1})](k),\\
\frac{2\pi i k}{T}\hat{\beta}_m(k)&=& [\hat{\beta}_m\ast(\hat{\beta}_2-\hat{\beta}_m)\ast(\hat{g}_x-\hat{1})](k),\\
\frac{2\pi i k}{T}\hat{\beta}_1(k)&=&[\hat{\beta}_1\ast((\hat{\beta}_s-\hat{\beta}_1)\ast\hat{g}_x+(\hat{\beta}_2-\hat{\beta}_m)\ast g_y-(\hat{\beta}_s-\hat{\beta}_1))](k),\\
\frac{2\pi i k}{T}\hat{\beta}_2(k)&=&[\hat{\beta}_2\ast((\hat{\beta}_s-\hat{\beta}_1)\ast\hat{g}_x+(\hat{\beta}_2-\hat{\beta}_m)\ast g_y-(\hat{\beta}_2-\hat{\beta}_m))](k),\\
\end{array}
\ee
for all $k\in \Z$ where we have also used the notation $\hat{1}(0)=1$ and $\hat{1}(k)=0$ for $k\neq 0$ and employed the obvious definition for addition of infinite sequences.  
\end{prop}

\begin{proof}
To complete the proof we only have to recall another basic fact from Fourier analysis. For two $T$-periodic functions we $f$, $g$ we have
\beann
f(t)g(t)&=&\sum_{k=-\I}^\I\sum_{m=-\I}^\I \hat{f}(k)\hat{g}(m)~e^{\frac{2\pi i(k+m)t}{T}}\\
&=& \sum_{n=-\I}^\I\sum_{k=-\I}^\I \hat{f}(k)\hat{g}(n-k)~e^{\frac{2\pi int}{T}}=\sum_{n=-\I}^\I (\hat{f}\ast\hat{g})(n) ~e^{\frac{2\pi int}{T}}
\eeann 
This formula for Fourier coefficients of products of functions yields the right-hand side of equation \eqref{eq:mod_flow_Fourier} as a direct consequence of Theorem \ref{thm:mod_flow}. The left-hand side of equation \eqref{eq:mod_flow_Fourier} follows from direct differentiation which is allowed since all our periodic functions are assumed to be sufficiently smooth; see Section \ref{sec:Fourier}.
\end{proof}
 
\begin{figure}[htbp]
\psfrag{b1}{$\hat{\beta}_1$}
\psfrag{b2}{$\hat{\beta}_2$}
\psfrag{bs}{$\hat{\beta}_s$}
\psfrag{bm}{$\hat{\beta}_m$}
\psfrag{gx}{$\hat{g}_x$}
\psfrag{sx}{$\hat{s}_x$}
\psfrag{g1}{$\hat{\gamma}_1$}
\psfrag{g2}{$\hat{\gamma}_2$}
\psfrag{k}{$k$}
	\centering
		\includegraphics[width=1\textwidth]{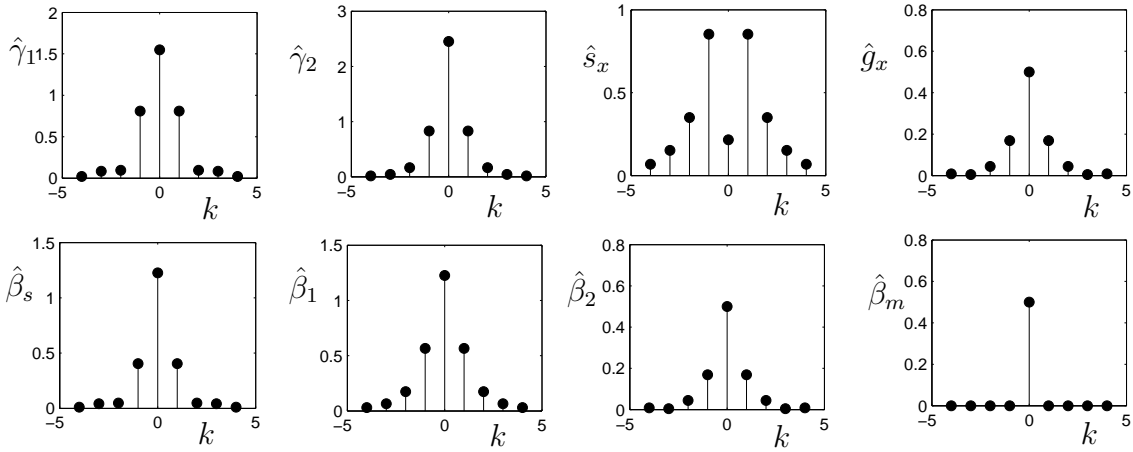}
	\caption{\label{fig:fig2}Absolute value of the first nine Fourier coefficients ($|k|\leq 4$) associated to the stable periodic orbit $\gamma(t)$ of \eqref{eq:RM}; the parameters are given in \eqref{eq:RM_para}. The coefficients of the phase space coordinates as well as the generalized elasticity and scale functions are shown.}
\end{figure}

Since \eqref{eq:mod_flow_Fourier} is an infinite set of algebraic equations it may look like we have not considerably simplified the problem of finding scale functions that are consistent with prescribed elasticities. However, the rapid decay of Fourier coefficients provided by Theorem \ref{thm:Fc_decay} allows us to approximate the solution of \eqref{eq:mod_flow_Fourier} by focusing on the first few harmonics with $|k|\leq \kappa\ll \I$. 

\begin{ex}[Example \ref{ex:ex1} continued]
Just for illustration purposes we look a the Fourier coefficients of generalized scale and elasticity functions in an example. Figure \ref{fig:fig2} shows the results for the Rosenzweig-MacArthur model from Section \ref{sec:specific} with $\kappa=4$. We can clearly see that the Fourier coefficients decay very rapidly; it is also interesting to observe that $\hat{s}_x(k)$ is bimodal for logistic growth whereas the other coefficients show a uni-modal distribution for the first few harmonics. For the Rosenzweig-MacArthur model the algebraic relations \eqref{eq:mod_flow_Fourier} on the Fourier coefficients become
\be
\label{eq:mod_flow_Fourier1}
\begin{array}{lcl}
\frac{2\pi i k}{T}\hat{\beta}_s(k)&=& [\hat{\beta}_s\ast(\hat{\beta}_s-\hat{\beta}_1)\ast(\hat{s}_x-\hat{1})](k),\\
\frac{2\pi i k}{T}\hat{\beta}_1(k)&=&[\hat{\beta}_1\ast((\hat{\beta}_s-\hat{\beta}_1)\ast\hat{g}_x+(\hat{\beta}_2-\hat{\beta}_m)-(\hat{\beta}_s-\hat{\beta}_1))](k),\\
\frac{2\pi i k}{T}\hat{\beta}_2(k)&=&[\hat{\beta}_2\ast((\hat{\beta}_s-\hat{\beta}_1)\ast\hat{g}_x)](k).\\
\end{array}
\ee

\begin{figure}[htbp]
\psfrag{D1}{{$\widehat{\beta_s'}$}}
\psfrag{D2}{{$\widehat{\beta_1'}$}}
\psfrag{D3}{{$\widehat{\beta_2'}$}}
\psfrag{ci}{{$k$}}
	\centering
		\includegraphics[width=1\textwidth]{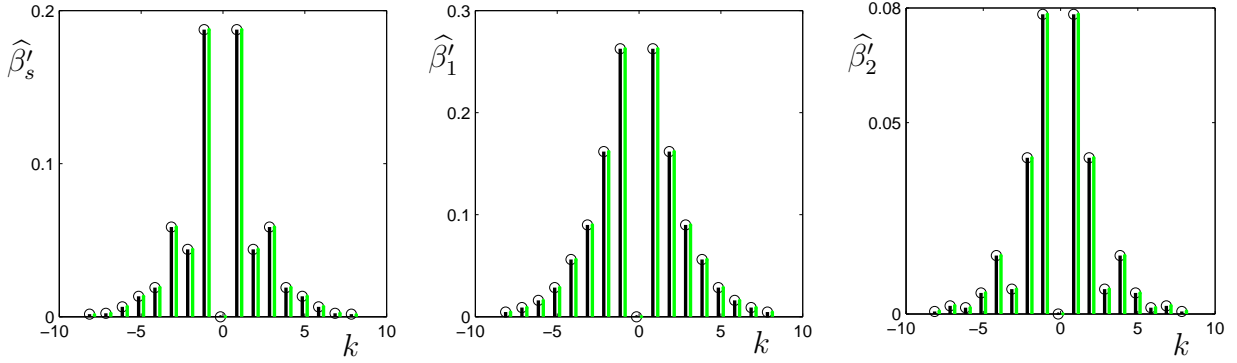}
	\caption{\label{fig:fig3}Absolute value of the first seventeen Fourier coefficients ($|k|\leq 17$) for the left-hand and right-hand sides of the algebraic conditions \eqref{eq:mod_flow_Fourier1}; parameter values used are given in \eqref{eq:RM_para}. The black coefficients (lines shifted slightly left) are the coefficients of the derivatives $\beta_s'$, $\beta_1'$ and $\beta_2'$ and the green coefficients (lines shifted slightly right) are associated to the periodic functions on the right-hand side of \eqref{eq:mod_flow_Fourier1}. The agreement of the two sets of coefficients is clearly visible.}
\end{figure}

Figure \ref{fig:fig3} shows the values of the Fourier coefficients for the Rosenzweig-MacArthur example where we see that the algebraic conditions \eqref{eq:mod_flow_Fourier1} are satisfied as proven in Proposition \ref{prop:FC_sat}. Furthermore, it is evident that due to the convolution a wider support $\kappa_M$ is necessary i.e.{} the algebraic equations \eqref{eq:mod_flow_Fourier} must be satisfied for $|k|\leq \kappa_M$ where $\kappa_M>\kappa$ and $\kappa$ is our truncation for the Fourier coefficients of the phase space periodic orbit.  
\end{ex}

\section{Stability Analysis}
\label{sec:stab1}

To get a better understanding of stability we can also use the Fourier series approach to re-express the Floquet multiplier \eqref{eq:FM1} given by
\be
\label{eq:FM2}
\lambda= \exp\left(\int_0^T \beta_s (s_x-1)-\beta_1 (g_x-1) +\beta_2(g_y-1)-\beta_m(m_y-1)dt\right).
\ee
The next results shows how the different Fourier coefficients enter in formula \eqref{eq:FM2}.

\begin{thm}
\label{thm:Floquet}
For the non-equilibrium generalized predator-prey model with $g_y=1=m_y$ the single Floquet multiplier of a $T$-periodic orbit is given by
\be
\label{eq:FM_Floquet}
\lambda=\exp\left(T\left(\underbrace{[\hat{\beta}_s\ast(\hat{s}_x-\hat{1})](0)}_{=:C_1}-\underbrace{[\hat{\beta}_1\ast(\hat{g}_x-\hat{1})](0)}_{=:C_2}\right)\right)=\exp(T(C_1-C_2))
\ee
i.e. whether $|\lambda|>1$ or $|\lambda|<1$ depends only on the difference of two zeroth-order Fourier coefficients $C_1$ and $C_2$ that arise from two discrete convolutions.
\end{thm}

\begin{proof}
We start by looking at the first summand in integral in \eqref{eq:FM2} which gives
\beann
\int_0^T \beta_s (s_x-1)dt&=&\int_0^T \sum_{k=-\I}^\I [\hat{\beta}_s\ast(\hat{s}_x-\hat{1})](k) ~e^{2\pi ikt/T}dt\\
&=&\sum_{k=-\I}^\I [\hat{\beta}_s\ast(\hat{s}_x-\hat{1})](k) \int_0^T e^{2\pi ikt/T}dt\\
&=&\left\{\begin{array}{ll}0 & \text{ for $k\neq 0$}\\ T ~[\hat{\beta}_s\ast(\hat{s}_x-\hat{1})](0) & \text{ for $k=0$}\\ \end{array}\right.
\eeann
where the last step follows from the fact that $\int_0^Te^{2\pi ikt/T}dt=0$ for $k\neq 0$. From this calculation we find $C_1$ and in a similar way also $C_2$. Using the two factors $C_{1,2}$ and $m_y=1=g_x$ in \eqref{eq:FM2} the result \eqref{eq:FM_Floquet} follows. Since $T>0$ the modulus $|\lambda|$ only depends on the difference of $C_1$ and $C_2$; if $C_1-C_2>0$ then $|\lambda|>1$ and if $C_1-C_2<0$ we obtain $|\lambda|<1$.
\end{proof}

Before we consider in more detail the dependency of stability on $C_1$ and $C_2$ we briefly investigate the influence of the period $T$. Although $T$ does not effect the stability of a periodic orbit directly it does have an interesting biological interpretation. If $T\gg 1$ then the period amplifies stability and instability. For example, when $C_1-C_2>0$ then a long period moves the multiplier even further away from $|\lambda|=1$ and trajectories near the unstable periodic orbit will escape vary quickly. On the other hand, if $C_1-C_2<0$ and $|\lambda|<1$ then a very large period $T$ moves the multiplier even closer to the super-attracting case $0\leq|\lambda|\ll1$. A very short period $0<T\ll 1$ has the effect of moving the multiplier very close to $|\lambda|\approx 1$. This means that when the periodic orbit is unstable, it will take a very long time to escape from it. The last effect can be interpreted as inducing meta-stability {i.e.} when the period of the predator-prey cycle is short then the predator-prey system stays near a metastable state for a long time although it is eventually unstable. This could lead to the conjecture that fast oscillations could be beneficial to survival for predator-prey populations during periods when external  parameters entering $C_1$ and $C_2$ drive the system, potentially only temporarily, to a state when $|\lambda|>1$. 

As a next step, we want to understand better how the Fourier coefficients of $\beta_s$, $\beta_1$, $s_x$ and $g_x$ influence $C_1$ and $C_2$.

\begin{prop}
\label{prop:C1C2}
The two constant $C_{1,2}$ are given by
\bea
C_1&=&\hat{\beta}_s(0)(\hat{s}_x(0)-1)+2\sum_{k=1}^\I (\text{Re}[\hat{\beta}_s(k)]\text{Re}[\hat{s}_x(k)]+\text{Im}[\hat{\beta}_s(k)]\text{Im}[\hat{s}_x(k)]) \label{eq:C1}\\
C_2&=&\hat{\beta}_1(0)(\hat{g}_x(0)-1)+2\sum_{k=1}^\I (\text{Re}[\hat{\beta}_1(k)]\text{Re}[\hat{g}_x(k)]+\text{Im}[\hat{\beta}_1(k)]\text{Im}[\hat{g}_x(k)]) \label{eq:C2}
\eea
\end{prop}

\begin{proof}
Given two sequences $\{\hat{f}(k)\}_{k=-\I}^\I$ and $\{\hat{g}(k)\}_{k=-\I}^\I$ of Fourier coefficients for two real-valued functions a direct calculation yields
\beann
(\hat{f}\ast \hat{g})(0)&=&\sum_{k=-\I}^\I \hat{f}(k)\hat{g}(-k)\\
&=&\hat{f}(0)\hat{g}(0)+\sum_{k>0} \hat{f}(-k)\hat{g}(k)+\sum_{k>0} \hat{f}(k)\hat{g}(-k)\\
&=&\hat{f}(0)\hat{g}(0)+\sum_{k>0} \overline{\hat{f}(k)}\hat{g}(k)+\sum_{k>0} \hat{f}(k)\overline{\hat{g}(k)}
\eeann
where we have used $\hat{f}(-k)=\overline{\hat{\bar{f}}(k)}$ and the real-valuedness $\bar{f}=f$, $\bar{g}=g$ in the last step. Next, observe that 
\benn 
\overline{\hat{f}(k)}\hat{g}(k)+ \hat{f}(k)\overline{\hat{g}(k)} =2(\text{Re}[\hat{f}(k)]\text{Re}[\hat{g}(k)]+\text{Im}[\hat{f}(k)]\text{Im}[\hat{g}(k)]).
\eenn
Now the formulas \eqref{eq:C1}-\eqref{eq:C2} follow immediately as $\beta_s(t)$, $\beta_1(t)$, $s_x(t)$ and $g_x(t)$ are all real-valued.
\end{proof}

In practice, we never use the infinite sum formulas from Proposition \ref{prop:C1C2} but truncate them at a finite order. Using the explicit formulas for $C_1$ and $C_2$ we can directly draw several conclusions regarding periodic solutions depending on generalized scale and elasticity functions (recall: we still use $g_y=1=m_y$). If all Fourier coefficients of higher-order $k\geq 1$ are small, then stability of periodic solutions is dominated by the terms
\benn
C_1\approx \hat{\beta}_s(0)(\hat{s}_x(0)-1)\qquad  \text{and}\qquad C_2\approx \hat{\beta}_1(0)(\hat{g}_x(0)-1).
\eenn 
Since the scale functions are always positive the time averages of the elasticity functions $\hat{s}_x(0)$ and $\hat{g}_x(0)$ determine the signs of $C_1$ and $C_2$. Therefore, average sub-linear elasticity $\hat{s}_x(0)<1$ and average super-linear conversion $\hat{g}_x(0)>1$ enhance stability. In ecological terms $0\leq \hat{s}_x(0)<1$ means that, on average, the prey growth should be limited by external factors and $\hat{s}_x(0)>1$ means that, on average, the predation rate should be sensitive to prey abundance; see also \cite{GrossFeudel1} for an interpretation of the generalized parameters for the equilibrium case. Both conditions make intuitive sense: if the prey grows without external limitation then solutions may be expected to diverge from a periodic solution and become unbounded while insensitivity of predation to prey growth could potentially drive a system to extinction. Of course, also the inverse relationships hold so that $\hat{s}_x(0)>1$ and $\hat{g}_x(0)<1$ act towards de-stabilization. In this context, the scale functions act as amplifiers. For example, if $C_1<0$ and $C_2>0$ then a large average growth rate $\hat{\beta}_s(0)\gg 1$ and a large conversion rate $\hat{\beta}_1(0)\gg 1$ will enhance stability even more since the Floquet multiplier moves closer to the super-attracting regime $\lambda \approx 0$. In this case, initial conditions will be attracted much quicker to a stable limit cycle. If
\benn
\hat{\beta}_s(0)(\hat{s}_x(0)-1)-\hat{\beta}_1(0)(\hat{g}_x(0)-1)\approx 0 \qquad \text{i.e.} \quad \hat{\beta}_s(0)(\hat{s}_x(0)-1)\approx \hat{\beta}_1(0)(\hat{g}_x(0)-1)
\eenn
the leading-order terms between growth and predation balance and the stability properties are dominated by higher-order harmonics. The leading-order terms also become irrelevant for elasticity functions which average close to one 
\benn
\hat{s}_x(0)\approx 1 \qquad \text{and}\qquad \hat{g}_x(0)\approx 1.
\eenn

\begin{figure}[htbp]
\psfrag{s}{\scriptsize{$\hat{s}_x(1)$}}
\psfrag{s1}{\scriptsize{$\hat{s}_x(2)$}}
\psfrag{b1}{\scriptsize{$\hat{\beta}_s(1)$}}
\psfrag{b11}{\scriptsize{$\hat{\beta}_s(2)$}}
\psfrag{Re}{\scriptsize{Re}}
\psfrag{Im}{\scriptsize{Im}}
\psfrag{a}{\scriptsize{(a)}}
\psfrag{b}{\scriptsize{(b)}}
\psfrag{c}{\scriptsize{(c)}}
	\centering
		\includegraphics[width=1\textwidth]{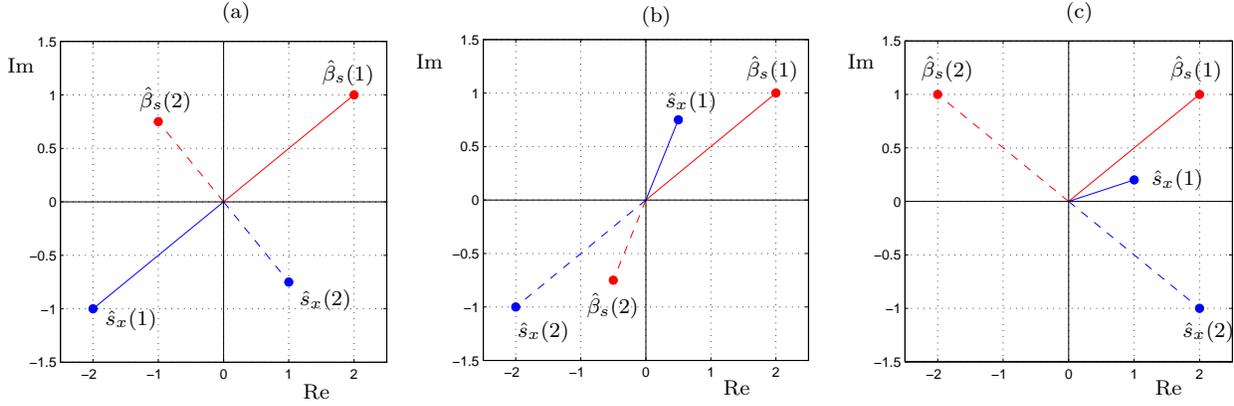}
	\caption{\label{fig:fig4}Illustration how the location of Fourier coefficients for generalized elasticity and scale functions influence stability. Here we focus on the first two higher-order harmonics (first coefficient = solid line, second coefficient = dashed line) of $\beta_s(t)$ (red) and $s_x(t)$ (blue) which influence the term $C_1$ in Theorem \ref{thm:Floquet}. (a) $\pi$-phase shift gives $C_1<0$. (b) Small phase shift gives $C_1>0$. (c) Competition between first- and second-order harmonics.}
\end{figure}

In this scenario we have to focus on the relations between the higher-order Fourier coefficients of $\beta_s$ and $s_x$ as well as $\beta_1$ and $g_x$. Let us assume for simplicity that $\hat{s}_x(0)=1=\hat{g}_x(0)$ so that we can focus on the higher harmonics. Then stability enhancing conditions are
\beann
C_1&=&\sum_{k=1}^\I (\text{Re}[\hat{\beta}_s(k)]\text{Re}[\hat{s}_x(k)]+\text{Im}[\hat{\beta}_s(k)]\text{Im}[\hat{s}_x(k)]) <0,\\
C_2&=&\sum_{k=1}^\I (\text{Re}[\hat{\beta}_1(k)]\text{Re}[\hat{g}_x(k)]+\text{Im}[\hat{\beta}_1(k)]\text{Im}[\hat{g}_x(k)]) >0.
\eeann
Figure \ref{fig:fig4} depicts several different situations in the complex plane for the first two higher-order harmonics of $\hat{\beta}_s(k)$ and $\hat{s}_x(k)$ ($k=1,2$). In Figure \ref{fig:fig4}(a) the first two higher-harmonics are in ``anti-phase'' so that the angles between the coefficients are separated by $\pi$. This means that 
\benn
\text{Re}[\hat{\beta}_s(k)]\text{Re}[\hat{s}_x(k)]<0 \quad \text{and} \quad \text{Im}[\hat{\beta}_s(k)]\text{Im}[\hat{s}_x(k)]<0
\eenn
for $k=1,2$. In such a situation, we expect that $C_1<0$ by disregarding higher orders so that stability is enhanced. 

Figure \ref{fig:fig4}(b) shows the situation where there is only a small phase difference between the coefficients (``in-phase'') which gives 
\benn
\text{Re}[\hat{\beta}_s(k)]\text{Re}[\hat{s}_x(k)]>0 \quad \text{and} \quad \text{Im}[\hat{\beta}_s(k)]\text{Im}[\hat{s}_x(k)]>0.
\eenn
There is also a possible situation where a competition between the different order harmonics arises as illustrated in Figure \ref{fig:fig4}(c). We can now also give an ecological interpretation of these conditions. Stability $C_1<0$ is enhanced if $s_x(t)$ and $\beta_s(t)$ oscillate with a phase separation near $\pi$ which means that a period of high sensitivity of prey abundance should coincide with a period of low prey growth and vice versa. Note that these conditions also make sense intuitively and suggest that prey growth is most efficient if there is a small number of prey and there are no limiting factors from the environment. Similar considerations also apply to the stability enhancing condition $C_2>0$. A small phase separation between $\beta_1(t)$ and $g_x(t)$ increases stability of the predator-prey limit cycle. Observe that $g_x(t)$ can be interpreted as the dependence of predation on prey abundance and $\beta_1(t)$ as a predation rate (normalized by the total number of prey) \cite{GrossFeudel1}; the stability conditions mean that a high predation rate should coincide with a high dependence of predation on prey abundance. In other words, if the dominating factor to prey abundance is predation then it is good for the predator to hunt a lot to increase stability of the limit cycle. 

Note that although the conclusions stated above seem to be ``obvious'' in an ecological context, it is by no means clear how to prove them. That they can be obtained by an analysis of nonlocal generalized models underlines the applicability of the approach.

\section{Sampling}
\label{sec:sampling}

Recall that due to Proposition \ref{prop:local_specific} it was straightforward for equilibrium generalized models to choose a set of generalized parameters, just random sampling produces a set of parameters that is consistent with at least one specific model. Random sampling of generalized parameters has been exploited to correlate different aspects of the dynamical system to stability \cite{ReznikSegre,SteuerGrossSelbigBlasius}. For non-equilibrium systems we must certainly check the necessary condition from Corollary \ref{cor:thm1}. One possibility is the following algorithm which allows sampling of elasticity and scale functions:

\begin{enumerate}
 \item[(A1)] Prescribe a set of $T$-periodic elasticity functions by their Fourier coefficients. For simplicity we will always choose $T=1$ and assume $g_y=1=m_y$.
 \item[(A2)] Choose a truncation order $\kappa_M$ for the algebraic system \eqref{eq:mod_flow_Fourier1} so that the necessary condition reads
\benn
\begin{array}{lclcr}
0&=&-2\pi i k\hat{\beta}_s(k)+ [\hat{\beta}_s\ast(\hat{\beta}_s-\hat{\beta}_1)\ast(\hat{s}_x-\hat{1})](k)&=:&c_s(k),\\
0&=&-2\pi ik \hat{\beta}_1(k)+[\hat{\beta}_1\ast((\hat{\beta}_s-\hat{\beta}_1)\ast\hat{g}_x+(\hat{\beta}_2-\hat{\beta}_m)-(\hat{\beta}_s-\hat{\beta}_1))](k)&=:&c_1(k),\\
0&=&-2\pi i k\hat{\beta}_2(k)+[\hat{\beta}_2\ast((\hat{\beta}_s-\hat{\beta}_1)\ast\hat{g}_x)](k)&=:&c_2(k),\\
\end{array}
\eenn
for $|k|\leq \kappa_M$.
 \item[(A3)] Define a new variable that collects all the Fourier coefficient values for $\beta_s$, $\beta_1$, $\beta_2$, $\beta_m$, $s_x$ and $g_x$
\benn
X:=(\hat{\beta}_s(0),\hat{\beta}_s(1),\ldots,\hat{\beta}_s(\kappa_m),\hat{\beta}_1(0),\ldots,\hat{\beta}_1(\kappa_m),\ldots)\in\C^{6(\kappa_M+1)}\cong \R^{12(\kappa_M+1)}
\eenn
where $x$ contains all the information about the scale and elasticity functions since the negative index coefficients can be obtained by complex conjugation.
\item[(A4)] Define a function
\benn
F(X):=\|\text{Re}(c_s)\|^2+\|\text{Re}(c_1)\|^2+\|\text{Re}(c_2)\|^2+\|\text{Im}(c_s)\|^2+\|\text{Im}(c_1)\|^2+\|\text{Im}(c_2)\|^2
\eenn
where we view $c_s$, $c_1$ and $c_2$ as vectors of dimension $2\kappa_M+1$ and real and imaginary parts are applied component-wise.
\item[(A5)] Observe that $F(X_0)=0$ if and only if the Fourier coefficients encoded in $X_0$ satisfy the algebraic equations in (A2). Therefore we can attempt to solve the following optimization problem
\be
\label{eq:optim}
X_m:=\min\{F(X):X\in \R^{12(\kappa_M+1)}\}
\ee
with a random initial condition, say $x=x_l$.
\end{enumerate}

Solving the optimization problem for different random initial conditions is expected to yield different values for $X_m$ that solve the algebraic constraint in (A2). This means that we get a set of Fourier coefficients $\{X_m(l)\}_{l=1}^L$ where $L$ denotes the sample size and the index $l\in\N$ indicates the dependence on the initial condition.

The main technical difficulty of the algorithm (A1)-(A5) is that it involves the solution of the optimization problem \eqref{eq:optim}. This is computationally much more expensive than the direct random sampling for equilibrium generalized models. It is known \cite{NocedalWright} that the main computational cost in optimization is often given by the difficulty of the function evaluations of $F(x)$. For our case, this seems to be the case since we have to compute several discrete convolutions to evaluate $F(x)$. However, the convolution computation is inexpensive due to the Fast Fourier Transform \cite{Koerner}.\\

\begin{figure}[htbp]
\psfrag{Ns}{$N_s$}
\psfrag{Nu}{$N_u$}
\psfrag{V}{$\quad$}
\psfrag{Rc1}{$\text{Re}(\hat{\beta}_s(1))$}
\psfrag{Rc2}{$\text{Re}(\hat{\beta}_s(2))$}
\psfrag{Ic1}{$\text{Im}(\hat{\beta}_s(1))$}
\psfrag{Ic2}{$\text{Im}(\hat{\beta}_s(2))$}
	\centering
		\includegraphics[width=1\textwidth]{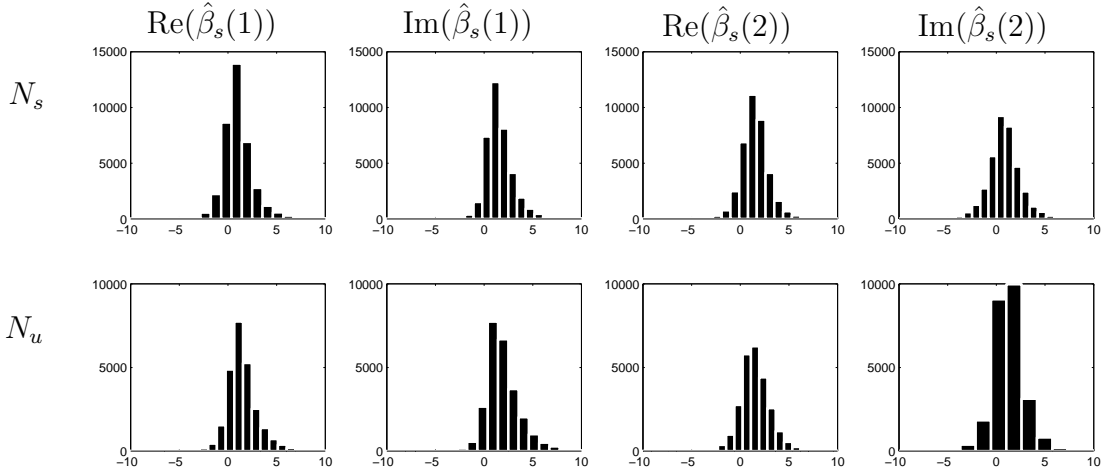}
	\caption{\label{fig:fig5}Histogram of the $5\cdot 63587$ Fourier coefficients for $\beta_s$ obtained from the optimization of \eqref{eq:optim} with uniformly sampled initial conditions \eqref{eq:ic_beta}. The columns show the five different real numbers with their observed number on the vertical axes. The first row shows coefficients associated to a stable Floquet multiplier and the second row those with an unstable Floquet multiplier. Observe that the number of stable coefficients is substantially larger than the number of unstable ones.}
\end{figure}

Now we want to demonstrate that the algorithm can be used for a sampling analysis of stability similar to the one used in \cite{GrossRudolfLevinDieckmann}. Let us point out that we do not attempt a full detailed statistical analysis here but that we only aim at a proof-of-principle. We solved \eqref{eq:optim} for 110000 initial conditions for $\kappa_M=2$ using a standard algorithm for nonlinear optimization \cite{MatLab2010b,LagariasReedsWrightWright}. Each sequence of Fourier coefficients in the initial condition consists of five real numbers e.g. 
\be
\label{eq:ic_beta}
\hat{\beta}_s(0), ~\text{Re}(\hat{\beta}_s(1)), ~\text{Im}(\hat{\beta}_s(1)), ~\text{Re}(\hat{\beta}_s(2)), ~\text{Im}(\hat{\beta}_s(2)),
\ee
which were sampled uniformly and independently from the interval $[0.5,1.5]$. We discarded all solutions of the optimization algorithm that did not satisfy the positivity condition
\benn
\hat{\beta}_s(0)>0,~\hat{\beta}_1(0)>0,~\hat{\beta}_2(0)>0,~\hat{\beta}_m(0)>0. 
\eenn
which is required by the definition of the scale functions and the invariance of the positive quadrant for the moduli space flow. The $63587$ remaining solutions $x_m(l)$ satisfied the the optimization problem (and therefore the moduli space flow) at least up to a tolerance of $10^{-4}$ i.e. $|x_m(l)|<10^{-4}$ for all $l$; the average value was $\mathbb{E}[x_m(l)]\approx 1.73\cdot10^{-6}$. We have also calculated the single Floquet multiplier $\lambda_l$ associated to each solution using Proposition \ref{prop:C1C2}.\\ 

\begin{table}[htb]
\begin{center}
\scriptsize{\begin{tabular}{l|c|c|c|c|c}  

$\hat{\beta}_s(k)$ & $\hat{\beta}_s(0)$ & $\text{Re}(\hat{\beta}_s(1))$ & $\text{Im}(\hat{\beta}_s(1))$ & $\text{Re}(\hat{\beta}_s(2))$ & $\text{Im}(\hat{\beta}_s(2))$ \\ 
\hline
mean (stable) & $1.6021$ & $0.2207$ & $0.5035$ & $-0.0117$ & $-0.1205$ \\
variance (stable) & $0.8840$ & $0.3226$ & $0.4000$ & $0.1000$ & $0.1202$ \\
mean (unstable) & $1.4409$ & $0.1936$ & $0.3863$ & $-0.0232$ & $-0.0591$ \\
variance (unstable) & $0.9129$ & $0.2863$ & $0.3409$ & $0.1014$ & $0.1381$ \\
\hline
& & & & &\\
$\hat{\beta}_1(k)$ & $\hat{\beta}_1(0)$ & $\text{Re}(\hat{\beta}_1(1))$ & $\text{Im}(\hat{\beta}_1(1))$ & $\text{Re}(\hat{\beta}_1(2))$ & $\text{Im}(\hat{\beta}_1(2))$ \\ 
\hline
mean (stable) & $1.3182$ & $0.3275$ & $0.3219$ & $0.0157$ & $0.2875$ \\
variance (stable) & $0.6575$ & $0.3050$ & $0.2431$ & $0.1502$ & $0.1226$ \\
mean (unstable) & $1.4502$ & $0.1778$ & $0.3445$ & $-0.0686$ & $0.2649$ \\
variance (unstable) & $0.7578$ & $0.3213$ & $0.2796$ & $0.1679$ & $0.1344$ \\
\hline
& & & & &\\
$\hat{\beta}_2(k)$ & $\hat{\beta}_2(0)$ & $\text{Re}(\hat{\beta}_2(1))$ & $\text{Im}(\hat{\beta}_2(1))$ & $\text{Re}(\hat{\beta}_2(2))$ & $\text{Im}(\hat{\beta}_2(2))$ \\ 
\hline
mean (stable) & $2.0805$ & $0.4823$ & $0.5761$ & $-0.0583$ & $0.2154$ \\
variance (stable) & $1.8865$ & $0.6716$ & $0.7486$ & $0.2135$ & $0.2970$ \\
mean (unstable) & $1.8914$ & $0.3399$ & $0.3835$ & $-0.0375$ & $0.0502$ \\
variance (unstable) & $1.8228$ & $0.5437$ & $0.5802$ & $0.1593$ & $0.2070$ \\
\hline
& & & & &\\
$\hat{\beta}_m(k)$ & $\hat{\beta}_m(0)$ & $\text{Re}(\hat{\beta}_m(1))$ & $\text{Im}(\hat{\beta}_m(1))$ & $\text{Re}(\hat{\beta}_m(2))$ & $\text{Im}(\hat{\beta}_m(2))$ \\ 
\hline
mean (stable) & $1.6184$ & $0.9865$ & $0.4756$ & $1.7548$ & $0.4965$\\
variance (stable) & $1.3568$ & $22090$ & $2.3338$ & $2.9748$ & $3.1220$ \\
mean (unstable) & $1.7736$ & $1.0465$ & $0.6171$ & $1.7785$ & $0.7635$ \\
variance (unstable) & $1.5189$ & $2.2642$ & $2.7503$ & $2.9247$ & $3.4595$ \\
\hline
& & & & &\\
$\hat{s}_x(k)$ & $\hat{s}_x(0)$ & $\text{Re}(\hat{s}_x(1))$ & $\text{Im}(\hat{s}_x(1))$ & $\text{Re}(\hat{s}_x(2))$ & $\text{Im}(\hat{s}_x(2))$ \\ 
\hline
mean (stable) & $1.5988$ & $1.0598$ & $1.6099$ & $1.5222$ & $0.8393$ \\
variance (stable) & $2.2079$ & $2.8519$ & $2.2995$ & $2.4361$ & $3.1797$ \\
mean (unstable) & $2.5967$ & $1.5343$ & $1.9850$ & $1.5559$ & $1.2697$ \\
variance (unstable) & $3.7412$ & $3.0413$ & $3.5637$ & $2.7521$ & $2.9364$ \\
\hline
& & & & &\\
$\hat{g}_x(k)$ & $\hat{g}_x(0)$ & $\text{Re}(\hat{g}_x(1))$ & $\text{Im}(\hat{g}_x(1))$ & $\text{Re}(\hat{g}_x(2))$ & $\text{Im}(\hat{g}_x(2))$ \\ 
\hline
mean (stable) & $2.7354$ & $1.9722$ & $2.4554$ & $0.9165$ & $1.3490$ \\
variance (stable) & $4.3300$ & $3.6009$ & $3.4094$ & $2.9125$ & $2.6612$ \\
mean (unstable) & $1.4787$ & $1.6774$ & $1.8789$ & $1.2302$ & $1.1109$ \\
variance (unstable) & $2.3056$ & $3.4147$ & $3.4502$ & $3.5286$ & $3.1522$ \\
\hline
\end{tabular}}
\caption{\label{tab:tab1}Mean and variance for the Fourier coefficients obtained from optimization (solution of the moduli space flow). Coefficients associated to stable and unstable Floquet multipliers are considered separately.}
\end{center}
\end{table}

Figure \ref{fig:fig5} shows some of the output of the computation. We plot the Fourier coefficients associated to the scale function $\beta_s$. The top row in Figure \ref{fig:fig5} corresponds to coefficients with stable periodic orbit ($|\lambda|<1$) and the bottom row to coefficients with an unstable periodic orbit ($|\lambda|>1$). We see that, despite the initial uniform sampling, the results for each coefficient of $\beta_s$ closely resemble normal distributions. The same observation also applies for the other scale and elasticity functions. In total we find that $37873$ solutions associated to a stable multiplier and $25714$ unstable ones. From this discrepancy one may either conjecture that the moduli space flow constraint could bias ecosystem towards stability or that our choice of initial uniform random sampling over a particular region in parameter space causes the bias towards stability. 

In Table \ref{tab:tab1} we list mean and variance for each coefficient. Several observations can be made based on Table \ref{tab:tab1}. The scale functions $\beta_2$ and $\beta_m$ have a much bigger variance than $\beta_s$ and $\beta_1$. This could indicate that the prey growth rate and the prey-per-capita predation rate have to obey much smaller ranges in ecosystems compared to the predator-per-capita rates describing consumption and mortality. It is also interesting that the mortality rate $\beta_m$ allows for much larger amplitude higher-order harmonics whereas {e.g.} $|\hat{\beta}_s(2)|$ is always comparatively small. The elasticities show no consistent variance decay towards higher-order harmonics although the coefficients themselves seem to decay. From the ecological perspective this suggest that predator-prey systems may exhibit a wide diversity in terms of sensitivities $s_x$ and $g_x$. 

To understand how the different coefficients relate to stability we calculate the Pearson correlation coefficient. For two vectors of observations $\{a_l\}$ and $\{b_l\}$ it is defined as
\benn
r(a,b):=\frac{\sum_l (a_l-\E[a])(b_l-\E[b])}{\sqrt{\sum_l (a_l-\E[a])^2\sum_l (b_l-\E[b])^2}}.
\eenn    

\begin{figure}[htbp]
\psfrag{bs0}{\scriptsize{$\hat{\beta}_s(0)$}}
\psfrag{bs1}{\scriptsize{$\hat{\beta}_s(1)$}}
\psfrag{bs2}{\scriptsize{$\hat{\beta}_s(2)$}}
\psfrag{bm0}{\scriptsize{$\hat{\beta}_m(0)$}}
\psfrag{bm1}{\scriptsize{$\hat{\beta}_m(1)$}}
\psfrag{bm2}{\scriptsize{$\hat{\beta}_m(2)$}}
\psfrag{b10}{\scriptsize{$\hat{\beta}_1(0)$}}
\psfrag{b11}{\scriptsize{$\hat{\beta}_1(1)$}}
\psfrag{b12}{\scriptsize{$\hat{\beta}_1(2)$}}
\psfrag{b20}{\scriptsize{$\hat{\beta}_2(0)$}}
\psfrag{b21}{\scriptsize{$\hat{\beta}_2(1)$}}
\psfrag{b22}{\scriptsize{$\hat{\beta}_2(2)$}}
\psfrag{gx0}{\scriptsize{$\hat{g}_x(0)$}}
\psfrag{gx1}{\scriptsize{$\hat{g}_x(1)$}}
\psfrag{gx2}{\scriptsize{$\hat{g}_x(2)$}}
\psfrag{sx0}{\scriptsize{$\hat{s}_x(0)$}}
\psfrag{sx1}{\scriptsize{$\hat{s}_x(1)$}}
\psfrag{sx2}{\scriptsize{$\hat{s}_x(2)$}}
\psfrag{r}{\scriptsize{$r$}}
\psfrag{bs1r}{\scriptsize{$\text{Re}\hat{\beta}_s(1)$}}
\psfrag{bs1i}{\scriptsize{$\text{Im}\hat{\beta}_s(1)$}}
\psfrag{bs2r}{\scriptsize{$\text{Re}\hat{\beta}_s(2)$}}
\psfrag{bs2i}{\scriptsize{$\text{Im}\hat{\beta}_s(2)$}}
	\centering
		\includegraphics[width=1\textwidth]{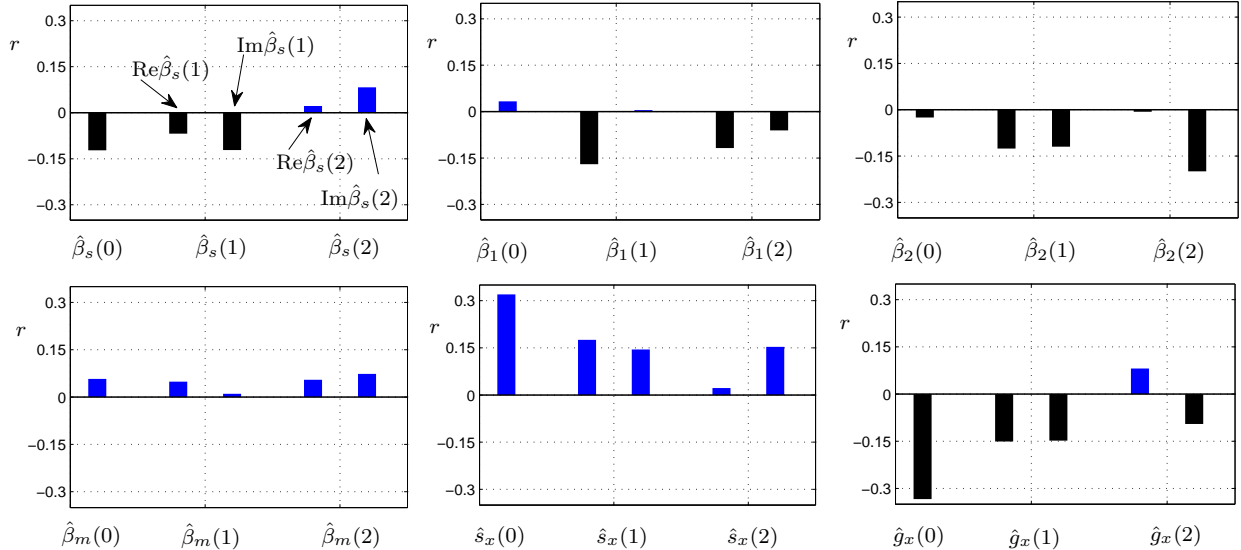}
	\caption{\label{fig:fig6}Pearson correlation coefficient $r=r(a,b)$ between stability and the different harmonics; positive correlation is indicated in blue and negative correlation in black. Each panel represents correlation for the Fourier coefficients from left to right. For example, the top left panel shows the values $r(\lambda_l,\hat{\beta}_{s,l}(0))$, $r(\lambda_l,\text{Re}(\hat{\beta}_{s,l}(1)))$, $r(\lambda_l,\text{Im}(\hat{\beta}_{s,l}(1)))$, $r(\lambda_l,\text{Re}(\hat{\beta}_{s,l}(2)))$, $r(\lambda_l,\text{Im}(\hat{\beta}_{s,l}(2)))$ from left to right where $\lambda_l$ is the Floquet multiplier with index $l$. }
\end{figure}

Figure \ref{fig:fig6} shows $r(a,\lambda_l)$ where $a$ is a sequence of real or imaginary parts of the Fourier coefficients e.g. $\{a_l\}=\{\text{Re}(\hat{\beta}_{s,l}(k))\}$. One important conclusion to draw from the correlation coefficients is that although a Fourier coefficient does not appear in the stability formula for the Floquet multiplier it may still correlate positively or negatively with stability. For example, $\hat{\beta}_2(1)$ and $\hat{\beta}_2(2)$ show a negative correlation with Floquet multiplier. This effect can be caused by the fact that the scale and elasticity functions are not independent {i.e.} they are related via the moduli space flow. 

It is very important to observe that we can recover conclusions, which we found already analytically in Section \ref{sec:stab1}, from the statistical analysis. For example, the coefficient $\hat{g}_x(0)$ correlates negatively with stability which means that decreasing it increases the Floquet multiplier and acts towards destabilization. This is precisely the result we have already obtained analytically in Section \ref{sec:stab1}. Let us point out again that the basic statistical analysis we have provided is incomplete but that it definitely does show that the proposed sampling techniques based on the FFT, optimization and correlations can help to understand stability of periodic solutions.

\section{Outlook}
\label{sec:extensions}

In this paper we have extended the method of generalized modeling from equilibrium to non-equilibrium systems. This extension has been achieved in the context of a classical predator-prey system with periodic solutions. The main re-scalings and definitions from the equilibrium case can be carried over to periodic orbits. However, the resulting generalized ODEs differ from the steady state case in several respects. The algebraic form is different due to the time dependent re-scalings and also the generalized parameters become time-dependent elasticity and scale functions. The Jacobian $A(t)$ of the system has to be analyzed using Floquet theory that describes the stability of periodic orbits. For planar vector fields we have been able to use Liouville's formula 
\benn
\lambda=\exp\left(\int_0^T Tr(A(t))dt\right)
\eenn 
which facilitated several analytical calculations. We have discovered that the generalized elasticity and scale functions have to satisfy a flow moduli space. Then we used Fourier analysis to find computable conditions from the moduli flow. Discrete convolutions turned out to be the key to stability analysis providing explicit interpretable stability results. In the last part of the paper, we suggested a sampling algorithm that uses optimization methods to find elasticity and scale functions that satisfy the (algebraic) moduli space flow. During our analysis we have also obtained several ecological conclusions about arbitrary predator-prey models that can be written in the generalized form \eqref{eq:gm_local}. 

In principle, we can extend the theory described here without any technical problems to limit cycles in $N$-dimensional systems for $N>2$; see also \cite{KuehnSiegmundGross} regarding generalizations to $\R^N$ in the equilibrium context. The main difference in $\R^N$ will be that we have to compute several the Floquet multipliers numerically since Liouville's formula only provides the product of the eigenvalues. 

One can also consider a generalization to non-equilibrium system beyond periodic orbits. For example, the generalized model \eqref{eq:GF2}, as well as Theorem \ref{thm:mod_flow} on the moduli space flow, carry over directly to other situations such as homoclinic trajectories \cite{Kuznetsov} or chaotic dynamics \cite{GH}. For instance, instead of posing periodic boundary conditions of the form 
\benn
\beta_1(0)=\beta_1(T), \qquad \beta_2(0)=\beta_2(T),\qquad \ldots 
\eenn
we have to impose other conditions on the solution of the moduli space flow. For homoclinic orbits we need the boundary conditions 
\benn
\beta_1(-\I)=\beta_1^*=\beta_1(\I), \qquad \beta_2(-\I)=\beta_2^*=\beta_2(\I), \qquad \ldots 
\eenn
{i.e.} that we have asymptotic limits of the generalized elasticity and scale functions to their value at a saddle-point equilibrium. For chaotic dynamics one must search for aperiodic bounded trajectories in moduli space. Note that this raises interesting mathematical as well as application questions. For example, the moduli space flow may provide new insights when a dynamical system may be chaotic. Finally, also our sampling analysis can obviously be extended. Beyond a more detailed statistical validation, we could consider higher-dimensional food webs \cite{GrossRudolfLevinDieckmann} which leads to a problem in $\R^N$.

%%%%%%%%%%%%%%%%%%%%%%%%%%%%%%%%%%%%%%%%%%%%%%%%%%%%%%%%%%%%%%%%%%%%%%%%%%%%%%%%%%%% 
\bibliographystyle{plain}
\bibliography{../my_refs}

\end{document}